\documentclass
{amsart}

\usepackage[latin1]{inputenc}

\usepackage{amsfonts}
\usepackage{latexsym}
\usepackage{amscd}

\usepackage{amsmath}
\usepackage{verbatim}
\usepackage{amsthm}
\usepackage{amssymb}
\usepackage{array}
\usepackage{enumerate}
\usepackage{longtable}

\usepackage{anysize}
\marginsize{2.5cm}{2.5cm}{3cm}{3cm}

\theoremstyle{plain}
\newtheorem{theorem}[subsection]{Theorem }
\newtheorem{definition}[subsection]{Definition }
\newtheorem{Pro}[subsection]{Proposition }
\newtheorem{remark}[subsection]{Remark}
\newtheorem{Cor}[subsection]{Corollary}
\newtheorem{corollary}[subsection]{Corollary}

\newtheorem{example}[subsection]{Example}
\newtheorem{lemma}[subsection]{Lemma}
\newtheorem{qu}[subsection]{Question}
\newtheorem{notation}[subsection]{Notation}
\newtheorem{conj}[subsection]{Conjecture}
 \numberwithin{equation}{section}
 
  \def\NB{{ }}
\def\Z{{\mathbb Z}}
\def\N{{\mathbb N}}
\def\R{{\mathbb R}}

\def\C{{\mathbb C}}

\def\S{{\mathbb S}}
\def\(SB){{\mathfrak B}}
\def\cont{\mathfrak c}
\begin{document}

\author{Dikran Dikranjan}
\address{Department of Mathematics and Computer Science, University of Udine, Via delle Scienze, 208-Loc. Rizzi, 33100 Udine, Italy}
\email{dikranja@dimi.uniud.it}

\author{Elena Mart\'in-Peinador}
\address{Departamento de Geometr\'{\i}a y Topolog\'{\i}a, Universidad Complutense de Madrid,
28040 Madrid, Spain}
\email{em\_peinador@mat.ucm.es}

\author{Vaja Tarieladze}
\address{Niko Muskhelishvili Institute of Computational Mathematics, 0171 Tbilisi, Georgia}
\email{vajatarieladze@yahoo.com}
\title{A class of metrizable locally quasi-convex groups which are not Mackey}

\thanks{The authors acknowledge support of the MICINN  of Spain through the grants MTM2009-14409-C02-01  for the first author and     MTM2009-14409-C02-02 (subprograma MTM)for  the second and third ones; V.Tarieladze was  also  partially supported  by
  GNSF/ST08/3-384 and GNSF/ST09/3-120.
The three authors also acknowledge to the IMI (Institute for Interdisciplinary Mathematics) for supporting several visits of Dikranjan and Tarieladze to the Complutense University of Madrid, which have allowed to carry out this joint research.}


\begin{abstract}  A topological group $(G,\mu)$ from a class  $\mathcal G$ of MAP topological abelian groups  will be called a {\it Mackey group} in $\mathcal G$
if it has the following property: if  $\nu$ is a  group topology in
$G$ such that $(G,\nu)\in \mathcal G$ and $(G,\nu)$ has  the same
continuous characters, say $(G,\nu)^{\wedge}=(G,\mu)^{\wedge}$, then
$\nu\le \mu$.

 If $\rm{LCS}$ is the class of Hausdorff topological abelian groups which admit a structure of a locally convex topological vector space over $\mathbb R$, it
is well-known that every metrizable  $(G,\mu) \in \rm{LCS}$ is a Mackey group in $\rm{LCS}$. For the class $\rm{LQC}$ of locally quasi-convex Hausdorff topological abelian groups,  it
 was proved in 1999
  that every {\bf complete} metrizable  $(G,\mu)\in \rm{LQC}$ is a Mackey group in $\rm{LQC}$ (\cite{CMPT}). The completeness cannot be \NB dropped within the class  $\rm{LQC}$ as we prove in this paper. In fact,
we provide  a  large family  of metrizable precompact \NB(noncompact)  groups which {\bf are not}  Mackey groups in \rm{LQC} (Theorem
\ref{basth}). Those examples  are constructed from groups of the form $c_0(X)$, whose elements are the null sequences of a
topological abelian group $X$, and whose topology is the uniform topology.
We first show that for a compact metrizable group $X\ne\{0\}$ the topological group $c_0(X)$ is a non-compact complete metrizable locally
quasi-convex group,  which has {\bf countable} topological dual iff $X$ is connected. Then we prove that for a connected
compact metrizable group $X\ne\{0\}$ the group $c_0(X)$ endowed with
 the product topology induced from the product $X^{\N}$ is metrizable precompact but not a Mackey group in  LQC.

\end{abstract}
\maketitle
\

\section{Introduction}
For a locally convex space $E$ there always exist a finest topology in the class of locally convex topologies giving rise to the same dual space $E^*$.
This topology was introduced  by Mackey and it was named after him.

A similar setting, as we describe below, can be considered for locally quasi-convex groups, a class which properly contains that of locally convex Hausdorff spaces. However the good results known to hold for the class of locally convex spces are no longer valid in this new class.

In this paper we change the point of view:  starting from sufficiently large classes of abelian groups, we define the corresponding  Mackey groups.
As will be seen, the theory is much reacher and properties like connectedness, local compactness, etc.
 play  an important role, which obviously was not the case  in the framework of locally convex spaces.

Let   $X,Y$ be groups.  We denote by $Hom(X,Y)$  the set of all group homomorphisms from
$X$ to $Y$. If $X,Y$ are topological groups,  $CHom(X,Y) $ stands for the continuous elements  of $Hom(X,Y)$.
\par
A set $\Gamma\subset Hom(X,Y)$ will be called separating, if
$$
(x_1,x_2)\in X\times X,\,x_1\ne x_2\,\Longrightarrow \exists \gamma\in \Gamma,\,\gamma(x_1)\ne\gamma (x_2)\,.
$$
For a  topological group $X$, a Hausdorff  group $Y$ and a non-empty $\Gamma\subset Hom(X,Y)$ we denote by $\sigma(X,\Gamma)$ the coarsest topology in $X$ with respect to which all members of $\Gamma$ are continuous. Note that $\sigma(X,\Gamma)$ is a group topology in $X$; moreover,
$\sigma(X,\Gamma)$ is Hausdorff iff $\Gamma$ is separating.
\par
The set
$$
\mathbb S:=\{s\in \mathbb C:|s|=1\}\,.
$$
is an abelian group with respect to multiplication of complex numbers; it is endowed with the usual topology induced from $\C$.
\par
From now on all considered groups will be abelian.
\par
For a group $G$ an element of $Hom(G,\mathbb S)$ is called a (multiplicative) {\it character}.

Let $G$ be  a topological  group.\\
 We write:
$$
G^\wedge:=CHom(G,\mathbb S)\,.
$$
An element of $G^\wedge$ is called {\it a continuous character}. Always $1\in G^\wedge$, where $1(x)=1\in \mathbb S,\,\,\forall x\in G$. The set  $G^\wedge$ with respect to pointwise multiplication of characters is  an abelian group with a  neutral element $1$.\\
The group $G^\wedge $ is called  {\it the topological  dual} of $G$.
{\it We shall not fix in advance a topology in}
$G^\wedge$.

A topological group $G$ is called {\it maximally almost periodic}, for short a MAP-group, if $G^\wedge $ is separating.
We denote by $\rm{MAP}$ also the class of all   MAP-groups.

For a (not necessarily discrete) topological group $G$ the topology $\sigma(G,G^{\wedge})$ is called {\it the Bohr topology} of $G$.
\par
For a group $G$ and for a group topology $\tau$ in $G$ we write $\tau^+$ for the Bohr topology of $(G,\tau)$. Clearly $\tau^+\le \tau$ and $\tau^+$
is a Hausdorff topology iff $G$ is MAP.

\begin{Pro}\label{coro}{\em (\cite{CoRo64}; cf. also \cite[Theorem 2.3.4]{DPS} and \cite{V})} For a  MAP-group $(G,\tau)$ the following statements hold:
\par
$(a)$ $\tau^+$ is a precompact group topology.
\par
$(b)$  $(G,\tau)^{\wedge}=(G,\tau^+)^{\wedge}$.
\par
$(c)$  $(G,\tau)$ is precompact iff  $\tau=\tau^+$.
\end{Pro}
\begin{definition}\label{cto}
Let $(G,\tau)$ be a topological group. A group topology $\eta$ in $G$ is said to be compatible for $(G,\tau)$ if $(G,\tau)^{\wedge}=(G,\eta)^{\wedge}$.
\end{definition}

The concept of {\it  compatible group topology } is due to Varopoulos \cite{V}.
 This remarkable paper offers, among other results, a description of {\it locally precompact} group topologies compatible with a MAP-group $(G,\tau)$.
  It is also proved there that  every  MAP-group $(G,\tau)$ admits  {\it at most one locally compact group topology in $G$ compatible with } $(G,\tau)$ \cite[p. 485]{V}. It is natural to consider the maximum (provided it exists) of all compatible topologies for  $(G,\tau)$, a  topological group in a certain class of groups. The following definition is in the spirit of \cite{BaKl}, where a categorical treatment  of Mackey groups is given.

\begin{definition}\label{def1} Let $\mathcal G$ be a class of  MAP-groups.
A topological group $(G,\mu)$  will be called a  {\it Mackey group} in $\mathcal G$ if $(G,\mu)\in \mathcal G$ and
if $\nu$ is a compatible group topology for $(G,\mu)$, with $(G,\nu)\in \mathcal G$, then $\nu\le \mu$.
 \end{definition}

\begin{definition}
Let $\mathcal G$ be a  class of MAP-groups and $(G,\nu)\in \mathcal G$. If there exists a group topology $\mu$ in $G$ compatible for
$(G,\nu)$ such that $(G,\mu)\in \mathcal G$  and  $(G,\mu)$ is  a Mackey group in $\mathcal G$, then $\mu$ is called the $\mathcal G$-Mackey topology in $G$ associated with $\nu$.
 \end{definition}

In this paper we will not discuss the problem of existence of $\mathcal G$-Mackey topologies (cf. \cite{DLMT}).
\par
Let $\rm{LPC}$ be the class of Hausdorff locally precompact topological abelian groups.
\par
A precompact $(G,\tau)\in \rm{LPC}$ may not be a Mackey group in $\rm{LPC}$ (for instance, if  $(G,\tau)$ is an infinite discrete group, then $(G,\tau^+)\in \rm{LPC}$ is not a Mackey group in
 $\rm{LPC}$. Clearly,  $(G,\tau)\in \rm{LPC}$, $\tau$ is compatible for $(G,\tau^+)$, but $\tau$ is strictly finer than $\tau^+$). However, the metrizability changes the picture as the following statement shows.

\begin{theorem}\label{varco}{\em (cf. \cite[Corollary 2 (p.484)]{V})}
Every metrizable  $(G,\mu)\in \rm{LPC}$ is a Mackey group in $\rm{LPC}$.
\end{theorem}

Next we consider  another class of groups in which the metrizable groups  are again Mackey.
\par
We say that a {\it topological} group $G$ {\it admits a structure of a (locally convex) topological vector space} over $\mathbb R$ if  there exists a map $\mathbb R\times G\to G$
making $G$ a  (locally convex) topological vector space over $\mathbb R$. It is known   that whenever  a Hausdorff topological
abelian group admits a topological vector space structure over $\mathbb R$, it must be unique.
\par
 Let $\rm{LCS}$ be the class of Hausdorff topological abelian groups which admit a structure of {\it locally convex} topological vector space over
$\mathbb R$. It is an important consequence of Hahn-Banach theorem that $\rm{LCS}\subset \rm{MAP}$. The next theorem is proved  in Section 2.

\begin{theorem}\label{mtv} Every metrizable  $(G,\mu)\in \rm{LCS}$ is a Mackey group in $\rm{LCS}$.
\end{theorem}

%
 \begin{remark}\label{kakol} {\em Let $\rm{MAPVS}$ be the class of  MAP-groups which admit a structure of  topological vector space over $\mathbb R$.
 It is known that $\rm{LCS}\subset \rm{MAPVS}$ and that this inclusion is strict.
 An analogue of Theorem \ref{mtv} fails for $\rm{MAPVS}$: {\it A metrizable  $(G,\mu)\in \rm{MAPVS}$ is a Mackey group
 in $\rm{MAPVS}$ iff $(G,\mu)$ is locally compact} (cf. \cite[Proposition 2.1]{CMPT}).}
 \end{remark}

  A natural class of groups containing    $\rm{LPC}\cup\rm{LCS}$ is provided by $\rm{LQC}$,  the class  of locally quasi-convex Hausdorff  groups (see Definition \ref{qcx}).  In fact, it is known that
 \begin{equation}\label{29seq}
 \rm{LPC}\cup\rm{LCS}\subset \rm{LQC}\subset
 \rm{MAP}
 \end{equation}
 where the inclusions are strict (cf. \cite{A, B, CMPT}).
It was proved in \cite {CMPT} that every {\bf Cech-complete}   $(G,\mu)\in \rm{LQC}$ is a Mackey group in $\rm{LQC}$.
In particular, every locally compact $(G,\mu)\in \rm{LQC}$ is a Mackey group in $\rm{LQC}$.
 We prefer to isolate here another particular case more relevant for the specific purposes of this paper:

\begin{theorem}\label{23seq}{\em (\cite {CMPT})}
Every {\bf complete metrizable }  $(G,\mu)\in \rm{LQC}$ is a Mackey group in $\rm{LQC}$.
\end{theorem}


In view of Theorems \ref{varco},\ref{mtv}, and \ref{23seq}, the following question arises:

\begin{qu}\label{MainQuestion}
{\em  Is every  metrizable} $(G,\mu) \in \rm{LQC}$ {\em a Mackey group in $\rm{LQC}$ }?
\end{qu}

We provide a negative answer to this question, in   fact we have:

\begin{theorem}\label{elena2}
Let $(G,\tau)\in$$\rm{LQC}$ be a non-precompact group with {\bf countable} dual  $(G,\tau)^{\wedge}$.
\par
Then $(G,\tau^+)$ is  a metrizable precompact group which is not a Mackey group in $\rm{LQC}$.
\end{theorem}

 In order to use Theorem \ref{elena2} and produce a large scale of counter-example to Question \ref{MainQuestion}, we need a general construction of
 non-precompact locally quasi-convex groups with countable dual (to which the theorem can be applied).   This goal is achieved my means of a class of groups which roughly speaking are 'groups of sequences'.   Let us   denote by $c_0(X)$ the subgroup of $X^\N$ of all  null sequences of $X$.
  The following holds: \\

 \noindent {\bf Theorem.} {\em Let $X$ be an infinite  compact metrizable abelian group,
  $\mathfrak{u}_0$ the uniform topology induced from $X^\N$ on $c_0(X)$.
  Then  $G := (c_0(X), \mathfrak{u}_0)$ is  a nonprecompact locally quasi-convex Polish group. Further, the following assertions are equivalent:
\begin{itemize}
     \item[(i)]  $X$ is connected.
   \item[(ii)]  $G^{\wedge}=(X^\wedge)^{(\N)}$.
     \item[(iii)]  $\mbox{{\rm Card} } G^{\wedge}=\aleph_0$.
     \item[(iv)]  $\mbox{{\rm Card} }G^{\wedge}< \cont$.
\end{itemize}
 }

To the proof of  this theorem (together with other more precise   results) are dedicated Sections 3--6.
The key point of the proof is the introduction of a class $\(SB)$ of groups $X$ (see Definition  \ref{SB}),
such that $c_0(X)^\wedge = (X^{\wedge})^{(\N)}$ for every precompact group $X\in \(SB)$ (Theorem \ref{NewTh}).
It turns out, that this new class $\(SB)$ can be described in some cases through well-known properties.
 Namely, in Corollary \ref{Coro_dikconj} (resp., Theorem \ref{dikTh}) we show that in the class of all locally compact (resp., metrizable) groups, $X \in \(SB)$ iff $X$ is connected (resp., locally generated in the sense of Enflo  \cite{Enf}, see Definition \ref{Def_loc_gen}).  This is the backbone of the proof of the theorem.

In the last section we offer some open questions and conjectures.


\subsection*{Notation and terminology} Let $\S$ denote the circle group and $\S_+=\{s\in \S:\rm{Re}(s)\ge 0\}\,.$
 $\cont$ will  stand for  the cardinality of continuum.

We denote  by $e$ the neutral element of a group. We also use the symbols  $0$ and  $1$ instead of $e$ if the group is known to be  additive or  multiplicative respectively.
 For a subset $A$ of a group $X$ denote by  $\langle A\rangle$ the subgroup of $X$
 generated by  $A$.

Let $X$ be a set. As usual, $X^\mathbb N$ will denote the set of all sequences ${\bf x}=(x_n)_{n\in \mathbb N}$ of elements of $X$ and $(p_n)_{n\in \N}$ the sequence of projections $X^\mathbb N\to X$.

For    a group $X$,   $X^{(\mathbb N)}$   will be  the subgroup of $X^\mathbb N$ consisting of all sequences   eventually equal to $e$. For $n\in \mathbb N$ define an injective homomorphism
$$
\nu_n: X \to  X^{(\mathbb N)}, \;\;  \mbox{  by } \;\; \nu_n(x)= (e,\dots,e, x,e,\dots),
$$
where $x\in X$ is placed in position $n$.\\
If $X$ is a topological group,  let

$$c_0(X):=\{ (x_n)_{n\in \mathbb N}\in  X^\mathbb N :\lim_nx_n=e\, \}.
$$

 Clearly $c_0(X)$ is a subgroup of $X^\mathbb N$ containing  $X^{(\mathbb N)}$; moreover,  $c_0(X)= X^{(\mathbb N)}$ iff $X$ has only trivial convergent sequences.

For a topological group $X$,   $\mathcal N(X)$ is the set of all neighborhoods of   $e \in X$. Clearly, $\S_+\in \mathcal N(\S)$.
We write $X^{\wedge}_{\rm{co}}$ for the group $X^{\wedge}$ endowed with the compact-open topology. For a subset $V$ of $X$
 let:
$$
V^{\triangleright}:=\{\xi\in X^{\wedge}:\xi(V)\subset \S_+\,\}\,.
 $$
If $X,Y$ are topological abelian groups  and $\varphi\in CHom(X,Y)$,  the mapping $\varphi^{\wedge}:Y^{\wedge}\to X^{\wedge}$,
 defined by  $\varphi^{\wedge}(\eta)=\eta \circ \varphi$  for $ \eta\in Y^{\wedge}$, is a group homomorphism called {\it the dual homomorphism}.

The {\em von Neumann's kernel } of a topological abelian group $X$ is defined by $${\mathbf n}(X)  = \bigcap \{\ker \xi: \xi \in X^\wedge\}.$$
 Clearly,  ${\mathbf n}(X)$  is a subgroup of  $X$, and $X$ is MAP iff ${\mathbf n}(X)=\{0\}$. If ${\mathbf n}(X)=X$ (i.e., $X^{\wedge}=\{1\}$), $X$ is called {\it minimally almost periodic}.


\section{Locally quasi convex groups}

Let us recall the definition of a locally quasi convex group.

\begin{definition}\label{qcx} {\em \cite{Vil}}
 A subset $A$ of a {\bf topological} group $G$ is called {\bf quasi-convex} if for every $x\in G\setminus A$ there exists
$\chi\in G^\wedge$ such that
$$
\chi(A)\subset \S_+,\,\,\,\text{but}\,\,\chi(x)\not\in \S_+\,.
$$
 A {\bf topological} group $G$ is called {\bf locally quasi-convex} if $\mathcal N(G)$ admits a basis consisting of quasi-convex subsets of $G$.
\end{definition}

Similar concepts were defined later in \cite{Su}, where  the terms {\it   polar set} and {\it   locally polar group} are used instead of   'quasi-convex set'  and   'locally  quasi-convex group'.  The author might not have been  aware of \cite{Vil}.

The locally  precompact groups are a prominent class of locally quasi-convex groups. The following  statement characterizes the groups $X\in \rm{LPC}$ with countable duals.

\begin{Pro}\label{codu1}
For an infinite locally precompact Hausdorff topological abelian group $X$ TFAE:
\par
(i) $X$ is precompact metrizable.
\par
(ii) $X^{\wedge}$ is countable.
\end{Pro}
\begin{proof}
$(i)\Longrightarrow (ii)$. This follows from \cite[(24.14)]{HR}.\\
$(ii)\Longrightarrow (i)$. Let $Y$ be the completion of $X$. It is known that the groups  $X^{\wedge}$ and  $Y^{\wedge}$ are algebraically isomorphic, hence, $Y^{\wedge}$ is countable. On the other hand, $Y$ is a locally compact Hausdorff topological abelian group therefore $Y^{\wedge}_{\rm{co}}$ is LCA. 
  Since  a second category countable Hausdorff  topological group is discrete, we have  that $Y^{\wedge}_{\rm{co}}$ is a discrete  countable group. Hence $(Y^{\wedge}_{\rm{co}})^{\wedge}_{\rm{co}}$ is a compact metrizable group. By Pontryagin's theorem, $Y$ and  $(Y^{\wedge}_{\rm{co}})^{\wedge}_{\rm{co}}$ are topologically isomorphic. Thus $Y$ is compact  metrizable and its topological subgroup $X$ is precompact metrizable.
\end{proof}

\begin{remark}{\em It follows from Proposition \ref{codu1} that if $X$ is a compact non-metrizable group, then $\rm{Card}(X^{\wedge})>\aleph_0$.}
\end{remark}

We shall see below that  implication $(ii)\Longrightarrow (i)$ of Proposition \ref{codu1} may fail if $X$ is a locally quasi-convex Hausdorff group.

\begin{Pro}\label{2oct}
Let $X$ be a precompact Hausdorff topological group and $V\in \mathcal N(X)$. Then $V^{\triangleright}$ is a finite subset of $X^{\wedge}$.
\end{Pro}

\begin{proof} It is known that $V^{\triangleright}$ is a compact subset of $X^{\wedge}_{\rm{co}}$.\\
 Suppose first that $X$ is compact.  Then $X^{\wedge}_{\rm{co}}$ is discrete and  in this case $V^{\triangleright}$ is finite.\\
If  $X$ is not compact, then it can be viewed as a dense subgroup of a compact Hausdorff topological group $K$. Let $U$ denote  the closure of $V$ in $K$; clearly $U\in \mathcal N(K)$. Since $\S_+$ is closed in $\S$, the density of $V$ in $U$ implies that
 $$
 V^{\triangleright}=\{\xi|_{X}: \xi\in U^{\triangleright}\,\}.
 $$
Now,   $U^{\triangleright}$  finite, implies that   $V^{\triangleright}$ is finite as well.
\end{proof}

Let us conclude this section with the proofs of Theorems \ref{mtv} and \ref{elena2}.

\medskip

\noindent{\bf Proof of Theorem \ref{mtv}.} Let $(G, \tau) \in \rm{LCS}$. Then  $(G, \tau)' $ will stand for the set of $\tau$-continuous linear forms from $G$ to $\R$. For every $l \in (G, \tau)'$ the mapping $\rho_l : G \to \S$ defined by $\rho_l (x) = exp \{2\pi i l(x) \}$ for all $x \in G$ is a continuous character. It is easy to see that the mapping   $\rho : (G, \tau)'
\to (G, \tau)^\wedge$, given by $l \mapsto \rho_l$  is an injective group homomorphism. It is surjective as well \cite[(23.32)]{HR}.
Therefore, we have:
\begin{equation} \label{smith}
(G,\tau)^{\wedge}=\{\rho_l: l\in (G, \tau)'\}\,.
\end{equation}
Take  a metrizable  $(G, \mu) \in$ LCS and let  $(G, \tau) \in$ LCS be such that $(G,\tau)^{\wedge}=(G,\mu)^{\wedge}$. Then from
 (\ref{smith}) we get $(G,\tau)^{'}=(G,\mu)^{'}$. Since $\mu$ is a metrizable locally convex vector topology in $G$, the last equality
 according to \cite[IV.3.4]{Schae} implies that $\tau\le \mu$ and  we are done. $\ \Box$\\

\medskip

\noindent{\bf Proof of Theorem \ref{elena2}.} Since $(G,\tau)^{\wedge}$ is countable, $\tau^+$ is metrizable. The topology $\tau^+$ is  precompact and compatible for $(G,\tau)$ by Proposition \ref{coro}. The group $(G,\tau^+)\in \rm{LQC}$ because precompact groups are locally quasi-convex. Since $(G,\tau)$ is not precompact we have that $\tau^+ < \tau$ being $\tau^+\ne \tau$ again by Proposition \ref{coro}. Hence $(G,\tau^+)$ is  a metrizable precompact group which is not a Mackey group in $\rm{LQC}$. $\ \Box$

\section{Groups of sequences}
\subsection{The uniform topology in $X^{\N}$.}


In what follows $X$ will be a fixed Hausdorff topological abelian group. We denote by $\mathfrak{p}_X$ the product topology in $X^{\N}$ and
by $\mathfrak{b}_X$ the box topology in $X^{\N}$.
\par
It is easily verified that the collection
$$
\{ V^{\N}:V\in \mathcal N(X)\}
$$
 is a basis at $e$ for a group topology in $X^{\N}$ which we denote by  $\mathfrak{u}_X$ and call {\it the uniform  topology}.
 In all three cases we shall omit the subscript
 $_X$ when no confusion is possible.

   The topology  $\mathfrak{u}$ in $X^{\N}$ is nothing else but the topology of uniform convergence on $\N$ when the elements of $X^{\N}$ are viewed
   as functions from $\N$ to $X$ and $X$ is considered as a uniform space with respect to its left (=right) uniformity.
   Since it  plays an important role in the sequel, we give in the next proposition an account of its main properties.

 We write:
$$
\mathfrak{p}_0:=\mathfrak{p}|_{c_0(X)}\;\; , \;\;\mathfrak{b}_0:=\mathfrak{b}|_{c_0(X)}\;\; \mbox{ and }\;\; \mathfrak{u}_0:=\mathfrak{u}|_{c_0(X)}.
$$

 \begin{Pro}\label{descr1} Let $(X,+)$ be a Hausdorff topological abelian group.
 \begin{itemize}
   \item[(a)]  The uniform topology $\mathfrak{u}$    is a Hausdorff group topology in $X^\mathbb N$ with $\mathfrak{p}\le \mathfrak{u}\le \mathfrak{b}$. Moreover,
     \begin{itemize}
         \item[(a$_1$)]  $\mathfrak{p}|_{X^{(\mathbb N)}}=\mathfrak{u}|_{X^{(\mathbb N)}}\,\Longleftrightarrow\, X=\{0\}$.
         \item[(a$_2$)] $\mathfrak{u}|_{X^{(\mathbb N)}}=\mathfrak{b}|_{X^{(\mathbb N)}} \Longrightarrow\,
          X\,\text{is\, a \rm{P}-group}\Longrightarrow\,\mathfrak{u}=\mathfrak{b}$; in particular, if $X$ is metrizable and $\mathfrak{u}|_{X^{(\mathbb N)}}=\mathfrak{b}|_{X^{(\mathbb N)}}$, then $X$ is discrete.
      \end{itemize}
   \item[(b)]  The passage from $X$ to $(X^\mathbb N,\mathfrak{u})$ preserves (sequential) completeness,  metrizability, MAP and local quasi-convexity.

   \item[(c)]   If $X\ne\{0\}$ , then:
      \begin{itemize}
        \item[(c$_1$)]  $(X^\mathbb N,\mathfrak{u})$ is not separable.
        \item[(c$_2$)]  $(X^{(\N)}, \mathfrak{u}|_{X^{(\N)}})$ is not precompact and hence, $(c_0(X),\mathfrak{u}_0)$ and  $(X^\mathbb N,\mathfrak{u})$ are  \NB neither precompact.
       \end{itemize}
\end{itemize}
   \end{Pro}

  \begin{proof}
  $(a)$ The first assertion has a straightforward proof. 
 \par
 (a$_1$) Suppose that $X\ne \{0\}$. Take $x\in X\setminus\{0\}$. Then $\nu_k(x)\in X^{(\mathbb N)},\, k=1,2,\dots$ and the sequence $(\nu_k(x))_{k\in \mathbb N}$ tends to $0$ in $\mathfrak{p}$. Since $X$ is Hausdorff, there is a $V\in \mathcal N(X)$ such that $x\not\in V$. Then $\nu_k(x)\not\in V^\N,\, k=1,2,\dots$. Hence, the sequence $(\nu_k(x))_{k\in \mathbb N}$ does not tend to $0$ in $\mathfrak{u}$.
 \par
 (a$_2$) Suppose that $\mathfrak{u}|_{X^{(\mathbb N)}}\ge\mathfrak{b}|_{X^{(\mathbb N)}}$. Take arbitrarily $U_n\in \mathcal N(X),\,n=1,2,\dots$. Then $(\prod_{n\in \N}U_n)\cap X^{(\mathbb N)} $ is a neighborhood of zero in $\mathfrak{b}|_{X^{(\mathbb N)}}$. As $\mathfrak{u}|_{X^{(\mathbb N)}}\ge\mathfrak{b}|_{X^{(\mathbb N)}}$, there is a $V\in \mathcal N(X)$ such that
 $$
 V^\N\cap X^{(\mathbb N)}\subset (\prod_{n\in \N}U_n)\cap X^{(\mathbb N)}\,.
 $$
 From  $\nu_k(V)\subseteq V^\N\cap X^{(\mathbb N)},\,k=1,2,\dots$, we get: $\nu_k(V)\subseteq (\prod_{n\in \N}U_n)\cap X^{(\mathbb N)},\,k=1,2,\dots$.
 So, $V\subseteq U_n,\,n=1,2,\dots$, therefore $V\subset\bigcap_{n\in \N}U_n$. Thus,
 for each sequence $U_n\in \mathcal N(X),\,n=1,2,\dots$,  $\bigcap_{n\in \N}U_n\in \mathcal N(X)$. Consequently, $X$ is a P-group.\\
 The implication  '$X\,\text{is\,a P-group}\Longrightarrow\,\mathfrak{u}=\mathfrak{b}$' is easy to verify.

 The last assertion in (a$_2$) follows from the well-known fact that a metrizable P-space is discrete.

(b)    We omit the standard proofs of  the first two cases.

Assume that $X$ is MAP and  let  ${\bf x}=(x_n)_{n\in \mathbb N}\in  X^\mathbb N\setminus {\{\bf 0}\}$. Take $n\in \mathbb N$ such that $p_n({\bf x})\ne 0$.
   Since $X$ is MAP, there is $\xi\in X^{\wedge}$ such that $\xi(p_n({\bf x}))\ne 1$. Clearly $p_n$ is $\mathfrak{u}$-continuous, therefore $\varphi:=\xi\circ p_n\in (X^\mathbb N,u)^{\wedge}$ and $\varphi ({\bf x})\ne 1$. Hence, $(X^\mathbb N,\mathfrak{u})$ is MAP.

In order to prove that $(X^\mathbb N,\mathfrak{u})$ is locally quasi-convex provided that   $X$ has the same property,
just observe that for any quasi-convex    $V\in\mathcal N(X)$, $V^{\N}=\bigcap_{n\in \mathbb N}p_n^{-1}(V)\, $,
and quasi-convexity is preserved under  inverse images  \NB by continuous homomorphisms and under  arbitrary intersections.

(c$_1$) Let us fix $x\in X\setminus\{0\}$ and $U\in\mathcal N(X)$ such that $U\cap \{-x,x\}=\emptyset$. Write $C=\{0,x\}^{\mathbb N}$. Then
  \begin{equation}\label{tir12apr}
  {\rm{card}}(C)=\cont\,,\quad {\bf y_1}\in C,\,{\bf y_2}\in C,\,{\bf y_1}\ne {\bf y_2}\,\Longrightarrow\,{\bf y_1}- {\bf y_2}\not\in  U^{\N}\,.
  \end{equation}
  Take a symmetric $V\in\mathcal N(X)$ such that $V+V\subset U$. Let $D$ be a dense subset of $(X^\mathbb N,u)$. Then for every ${\bf y}\in C$ we can find a
  ${\bf d}_{{\bf y}}\in D$ such that ${\bf d}_{{\bf y}}\in V+{\bf y}$. From (\ref{tir12apr}) we get:
  $$
  {\bf y_1}\in C,\,{\bf y_2}\in C,\,{\bf y_1}\ne {\bf y_2}\,\Longrightarrow\,{\bf d}_{{\bf y}_1}\ne {\bf d}_{{\bf y}_2}\,.
  $$
 Thus, for any dense subset $D$,  
 ${\rm{card}}(D)\ge \cont$, and therefore $(X^\mathbb N,\mathfrak{u})$ is nonseparable.\\

 (c$_2$) Fix $x\in X$, $x\ne 0$ and  a symmetric $V\in \mathcal N(X)$ such that $x\not\in V$. Then
  \begin{equation}\label{cauchy}
  m,n\in \mathbb N,\,m\ne n\,\Longrightarrow\, \nu_m(x)-\nu_{n}(x)\not\in  V^\N\,.
  \end{equation}
 Now  (\ref{cauchy}) together with  $ \nu_m(x)-\nu_n(x)\in X^{(\N)}\ ,\forall m,n\in \N$, yield that $(X^{(\N)}, \mathfrak{u}|_{X^{(\N)}})$ is not precompact.
  \end{proof}
 \begin{remark}\label{NewUD}  {\em If $X$ is a compact metrizable group and $\rho$ is  an invariant metric for $X$,  then
 the equality
 $$d_{\infty}({\bf x},{\bf y})=\sup_{n\in \mathbb N} \rho (x_n, y_n),\quad {\bf x},{\bf y}\in {X}^\mathbb N\,
$$
defines an invariant metric for
 $(X^{\N},\mathfrak{u})$. In particular,
the topology of $({\mathbb S}^\mathbb N,\mathfrak{u})$ can be induced by the following metric: }
$$
d_{\infty}({\bf x},{\bf y})=\sup_{n\in \mathbb N}|x_n-y_n|,\quad {\bf x},{\bf y}\in {\mathbb S}^\mathbb N\,.\eqno(\dag)
$$
According to \cite[Example 4.2]{CoRo66} the metric group $({\mathbb S}^\mathbb N,d_{\infty})$ has the following remarkable property:
 it is not precompact, but every uniformly  continuous real-valued function defined on it, is bounded.

\end{remark}

  \subsection{The group of null sequences $c_0(X)$.}\par

 \NB
  In this section we study the group 
   of all null sequences of a  topological abelian group $X$ as a subgroup of $(X^{\N}, \mathfrak u)$.

  \begin{lemma}\label{agu1}
 Let   $X$ be a topological group. Then:
   \begin{itemize}
   \item[(a)]  $c_0(X)$ is closed in $(X^\mathbb N,\mathfrak{u})$.
   \item[(b)] If ${\bf x}=(x_n)_{n\in \mathbb N}\in c_0(X)$, then the sequence $\left(\sum_{k=1}^n\nu_k(x_k)\right)_{n\in  \N}$ converges to ${\bf x}$
   in the topology $\mathfrak{u}$;
   in particular, $X^{(\mathbb N)}$ is a $\mathfrak{u}$-dense subset of $c_0(X)$.
     \end{itemize}
  \end{lemma}

  \begin{proof} In order to prove $(a)$, pick ${\bf x}=(x_n)_{n\in \mathbb N} X^\N \setminus c_0(X )$. There exists then $V \in \mathcal N(X)$ and a subsequence $x_{n_k} \notin V$, $k \in \N$. Take a symmetric  $V_1\in\mathcal N(X)$ such that $V_1 + V_1 \subset V$.
  Now we have: $({\bf x} + (V_1)^{\N}) \cap c_0(X) =\emptyset$. In fact, for any ${\bf z} \in (V_1)^{\N}$, $x_{n_k} + z_{n_k}
\notin V_1 $ for otherwise $x_{n_k} \in V_1 + V_1 \subset V$. Therefore ${\bf x} + {\bf z} \notin c_0(X)$.

In order to prove $(b)$, let ${\bf y_n}=\sum_{k=1}^n\nu_k(x_k)$ for $n\in \N$. Fix $U\in\mathcal N(X)$ and pick a  symmetric $V\in\mathcal N(X)$ with $V\subset U$.
    Since ${\bf x}=(x_n)_{n\in \mathbb N}\in  c_0(X )$, for some $k_0\in \mathbb N$ we  have $x_k\in V,\,\forall k\ge k_0$, i.e., ${\bf y_k}- {\bf x} \in V^{\N}\subset  U^{\N}$. Hence ${\bf y_k} \in{\bf x} +  U^{\N}$ $\forall k\ge k_0$ and therefore ${\bf y_n}$  converges to ${\bf x}$. This proves also the last assertion of (b).
%
%
\end{proof}

 Thus, the situation described in the   previous Lemma is:
 $$X^{(\mathbb N)}\stackrel{\mathfrak{u}-densely}{ \subset } c_0(X)\stackrel{\mathfrak{u}-closed}{ \subset }X^\mathbb N$$
Since the groups of the form  $(c_0(X), \mathfrak{u}_0)$ are  the main object of our future considerations, we summarize now those properties inherited from the corresponding ones in $(X^{\N}, \frak u)$, or lifted from properties of $X^{(\mathbb N)}$.

   \begin{Pro}\label{descr2m} Let $X$ be a Hausdorff topological abelian group.
 \begin{itemize}
   \item[(a)]  $(c_0(X), \mathfrak{u}_0)$  is a Hausdorff topological group  having as  a basis at zero the collection $\{ V^{\N}\cap c_0(X) :V\in \mathcal N(X)\}$.
   \item[(b)]  $\mathfrak{p}_0 \le \mathfrak{u}_0\le \mathfrak{b}_0$. Moreover, $\mathfrak{p}_0 =\mathfrak{u}_0\,\Longleftrightarrow\, X=\{0\}$; if $X$ is metrizable and $\mathfrak{u}_0 =\mathfrak{b}_0$, then $X$ is discrete.
   \item[(c)]   The passage from $X$ to $(c_0(X),\mathfrak{u}_0)$ preserves (sequential) completeness, metrizability, separability,
 MAP, local quasi-convexity,  non-discreteness, and connectedness.
  \end{itemize}
  \end{Pro}

   \begin{proof}  (c)  Assume $X$ is separable.  The density of $X^{(\N)}$ in $c_0(X)$ yields that  $c_0(X)$ is also separable.\\
 If  $(c_0(X),\mathfrak u_0)$ is discrete, for some $V\in \mathcal N(X)$ we   have that $ V^\N\cap c_0(X)={\{\bf 0}\}$. Thus $V=\{0\}$ and  $X$ is discrete. \\
 The rest of (c) (except connectedness), as well as (a,b)  follows from Proposition  \ref{descr1}.

Assume now that $X$ is connected. Consequently, the product spaces $X^n$ are also connected, for all $n \in \N$.
Let
 $$G_n : = \{ {\bf x} \in  c_0(X) : x_k =0,  k= n+1, n+2, \dots  \}$$
 Then
 $(G_n, \mathfrak{u}_{|G_n})$ is topologically isomorphic to $(X^n, \mathfrak{p}_{|X^n})$, therefore
 connected.
 Since  $ X^{(\N)} = \bigcup_{n \in \N} G_n$  and $\bigcap_{n \in \N}G_n \neq \emptyset $, we obtain that  that $ (X^{(\N)}, \frak u_{|X^{(\N)}})$  and its closure 
  $(c_0(X),\mathfrak u_0)$ are  connected.
     \end{proof}

  \begin{remark}\label{Ro}{\em
 The metric group $(c_0(\mathbb S),d_{\infty})$ was introduced by Rolewicz in \cite{Rol}, where he proves that it is a monothetic group.  As he underlines,  a monothetic and completely metrizable group need not be compact  or discrete, a fact   which  breaks the dichotomy existing in the class of LCA-groups: namely,  a monothetic LCA-group must be  either compact or discrete (\cite[ Lemme 26.2 (p. 96)]{W}; see also \cite[Remark 5]{anz}, where a construction of a different example of a complete metrizable monothetic non-locally compact group is indicated).

A proof of the fact that $(c_0(\mathbb S),\mathfrak{u}_0)$ is monothetic  is contained in \cite[pp. 20--21]{DPS} (cf. also \cite{Gabr}, where it is shown further that $(c_0(\mathbb S),\mathfrak{u}_0)$ is Pontryagin reflexive).}
\end{remark}

   \begin{remark}\label{70ct1}{\em Let $X$ be the group $ \mathbb R $ with the usual topology.
\begin{itemize}
   \item[(1)] By Proposition \ref{descr1} $(\mathbb R^{\mathbb N},\mathfrak{u})$ is a complete metrizable topological abelian group. The group  $(\mathbb R^{\mathbb N},u)$ is not connected; the connected component of the null element coincides with $l_{\infty}$ and the topology $\mathfrak{u}|_{l_{\infty}}$ is the usual Banach-space topology of
 $l_{\infty}$. It follows  that although $\mathbb R^{\mathbb N}$ is a vector space over $\mathbb R$, the topological group $(\mathbb R^{\mathbb N},\mathfrak{u})$ {\bf is not} a topological vector space over $\mathbb R$.
  \item[(2)] By Proposition \ref{descr2m} (c), $(c_0(\mathbb R),\mathfrak{u}_0)$ is a complete separable metrizable \NB connected topological abelian group. Note that  $c_0(\mathbb R)$ is a vector space over $\mathbb R$ and  $(c_0(\mathbb R),\mathfrak{u}_0)$ {\bf is}  a topological vector space over $\mathbb R$. The topology $\mathfrak{u}_0$ is the usual Banach-space topology of $c_0$.
   \item[(3)] It is easy to see that $\mathbb Z^{(\mathbb N)}$ is a closed subgroup of  $(c_0(\mathbb R),\mathfrak{u}_0)$ and the quotient group
 $$
 (c_0(\mathbb R),\mathfrak{u}_0)/\mathbb Z^{(\mathbb N)}
 $$
  is topologically isomorphic with $(c_0(\mathbb S),\mathfrak{u}_0)$.
\end{itemize}
 }
   \end{remark}

    For an additive  topological abelian group  $X$ we introduce the following three subgroups of $X^\N$   included between $X^{(\mathbb N)}$ and $c_0(X)$:
  $$
{\rm cs}(X)=\{{\bf x}=(x_n)_{n\in \mathbb N}\in X^{\mathbb N}:\left(\sum_{k=1}^nx_k\right)_{n\in \mathbb N}\text{is\, a\, Cauchy\, sequence\, in }\, X\},
  $$
 $$
  {  \rm ss}(X)=\{{\bf x}=(x_n)_{n\in \mathbb N}\in X^{\mathbb N}:\left(\sum_{k=1}^nx_k\right)_{n\in \mathbb N}\text{is\, a\, convergent\, sequence\, in }\, X\,\}.
  $$
  and
  $$
  { l}(X)=\{{\bf x}=(x_n)_{n\in \mathbb N}\in X^{\mathbb N}:(x_{\sigma(n)})_{n\in \mathbb N}\in {  \rm ss}(X)\,\,\text{for\, every\, bijection}\,\sigma:\N\to\N\,\}.
  $$
  The same notation will be used if $X$ is  a multiplicative topological abelian group: in fact,   these three groups are defined similarly.

Clearly,
$$
 X^{(\mathbb N)}\subset{ l}(X)\subset {  \rm ss}(X)\subset {  \rm cs}(X)\subset c_0(X).
$$
It is easy to observe that for a Hausdorff topological abelian group $X$ the equality $\rm{ss}(X)=\rm{cs}(X)$ holds iff $X$ is sequentially complete.

The notation ${  \rm cs}(X)$ and ${  \rm ss}(X)$ are not standard, while
 $ {l}(X)$ can be justified as follows: usually $ l$ stands for the set of all real absolutely summable sequences and by Riemann-Dirichlet theorem,
\begin{equation}\label{rd1}
 { l}(\R)=\{{\bf x}=(x_n)_{n\in \mathbb N}\in \R^{\mathbb N}:(|x_n|)_{n\in \mathbb N}\in {  \rm ss}(\R)\}= l\,.
\end{equation}
Observe that we also have the following analogue of (\ref{rd1}) for $\S$ (cf. \cite[Ch. VIII.2, Theorem 1 (p.116)]{BourII}):
\begin{equation}\label{rd2}
 { l}(\S)=\{{\bf x}=(x_n)_{n\in \mathbb N}\in \S^{\mathbb N}:(|1-x_n|)_{n\in \mathbb N}\in {  \rm ss}(\R)\}\,.
\end{equation}
It is well known that $l(\R)\ne \rm{ss}(\R)\ne c_0(\R)$.

 Let us consider now the situation  in the general case.  First of all we note that if $X$ has only trivial convergent sequences, then
$$
 X^{(\mathbb N)}={ l}(X)={  \rm ss}(X)={  \rm cs}(X)= c_0(X)\,.
$$
However, for a group $X$ the equality $\rm{cs}(X)= c_0(X)$ need   not imply the equality $X^{(\mathbb N)}=c_0(X)$ as the following proposition shows.

  \begin{Pro}\label{alonso} Let $X$ be a topological abelian group.
\begin{itemize}
   \item[(a)] If $\mathcal N(X)$ admits a base consisting of subgroups of $X$, then ${  \rm cs}(X)= c_0(X)$.
   \item[(b)] If $X$ is sequentially complete and $\mathcal N(X)$ admits a base consisting of subgroups of $X$, then ${ l}(X)={  \rm ss}(X)= c_0(X)$.
   \item[(c)] If $X$ is totally disconnected and locally compact, then ${ l}(X)={  \rm ss}(X)= c_0(X)$.
\end{itemize}
  \end{Pro}

\begin{proof} $(a)$ is   easy to verify and $(b)$ follows from $(a)$. Finally, $(c)$ follows from $(b)$, since our hypothesis implies that $\mathcal N(X)$ admits a basis consisting of subgroups of $X$ (\cite[Theorem II.7.7 (p. 62)]{HR}). \end{proof}

We shall see  (Remark \ref{bourEr}) that if a non-trivial group $X$ is either connected and metrizable or connected and locally compact, then ${  \rm ss}(X)\ne c_0(X)$. {\it It is not clear whether for a complete metrizable  abelian group $X$ the equality ${ l}(X)={  \rm ss}(X)$  implies the equality ${  \rm ss}(X)= c_0(X)$}.

\section{The  $\beta$-dual  of a group of sequences }


 For  topological abelian groups  $X,Y$ and a non-empty $A\subset X^{\mathbb N}$  we write:
$$
   A^{\beta}(Y)=\{{\bf h}=(\xi_n)_{n\in \mathbb N}\in (CHom(X,Y))^{\mathbb N}: (\xi_n(x_n))_{n\in \mathbb N}\in ss(Y)\quad\forall {\bf x}=(x_n)_{n\in \mathbb N}\in A\}.
    $$
  The notation $A^{\beta}(Y)$ is taken from the theory of sequence spaces, where it is used for $X=Y=\R$ (see, e.g.,\cite{BeKa}).
  \par
  Instead of $A^{\beta}(\S)$ we will use the shorter notation $A^{\beta}$. In other words,
   $$
    A^{\beta}=\{{\bf h}=(\xi_n)_{n\in \mathbb N}\in (X^{\wedge})^{\mathbb N}: (\xi_n(x_n))_{n\in \mathbb N}\in ss(\S)\quad\forall {\bf x}=(x_n)_{n\in \mathbb N}\in A\}.
    $$
  Clearly, $A^{\beta}$ is a subgroup of $(X^{\wedge})^{\mathbb N}$ containing $(X^{\wedge})^{(\mathbb N)}$.
  \NB If $A$ is a subgroup, we will  call    $A^{\beta}$ the $\beta$-dual of $A$.
\par
Let $A\subset X^{\mathbb N}$ be a subgroup and $\tau$ a group topology in $A$. The main motivation to introduce 
\NB the $\beta$-dual  comes from the fact that,
under appropriate hypotheses, $(A,\tau)^{\wedge}$ is canonically isomorphic with  $A^{\beta}$. Now we describe several cases when this occurs, being $A : = c_0(X)$ the most important of them (See next section).

\begin{lemma}\label{6apr1}\label{tol2} Let $X$ be a topological abelian group. Then
\begin{equation}\label{tol1}
\left( X^{(\N)} \right)^{\beta}=(X^{\wedge})^{\mathbb N}\;\mbox{ and }\;\left( X^{\N} \right)^{\beta}=(X^{\wedge})^{(\mathbb N)}\,.
\end{equation}
\end{lemma}

\begin{proof} The first equality in (\ref{tol1}) is trivial. To prove the second one  it is sufficient to  show that $\left( X^{\N} \right)^{\beta}\subset (X^{\wedge})^{(\mathbb N)}\,.$ Take
\begin{equation}\label{tol4}
{\bf h}=(\xi_n)_{n\in \mathbb N}\in \left( X^{\N} \right)^{\beta}
\end{equation}
and let us see that the assumption
\begin{equation}\label{tol5}
{\bf h}=(\xi_n)_{n\in \mathbb N}\not\in (X^{\wedge})^{(\mathbb N)}
\end{equation}
  leads to a contradiction with (\ref{tol4}).\\
From (\ref{tol5}) we get the existence of a  strictly increasing sequence $(k_n)_{n\in \N}$ of natural numbers such that
\begin{equation}\label{tol6}
\xi_{k_n}(X)\ne\{1\},\,n=1,2,\dots
\end{equation}
 Since  $\S_+$ {\it contains no nontrivial subgroup of} $\S$, from (\ref{tol6}) we get the existence of a  sequence $(x_{k_n})_{n\in \N}$ of elements of $X$ such that
\begin{equation}\label{tol7}
\xi_{k_n}(x_{k_n})\not\in \S_+,\,n=1,2,\dots
\end{equation}
Put $x_j=0$ for all $j\in \N\setminus \{k_1,k_2,\dots\}$. For the sequence ${\bf x}=(x_n)_{n\in \mathbb N}\in X^{\N}$ obtained in this way,  the sequence
$(\prod_{k=1}^{n}\xi_k(x_k))_{n\in \mathbb N}$ converges in $\mathbb S\,$ by virtue of (\ref{tol4}).   This implies that $\lim_n\xi_n(x_n)=1$. Hence,
$\lim_n\xi_{k_n}(x_{k_n})=1$. However, as $\S_+\in \mathcal N(\S)$, the equality $\lim_n\xi_{k_n}(x_{k_n})=1$ contradicts (\ref{tol7}).
\end{proof}

In the following assertion an additive abelian group $A\subset X^{\mathbb N}$ is treated as a $\Z^\N$-module.

\begin{lemma}\label{4oct1}
Let $X$ be a topological abelian group, $A\subset X^{\mathbb N}$ a non-empty subset having the property:
\begin{equation}\label{6oct1}
\{0,1\}^{\N}A\subset A\,,
\end{equation}
$(\xi_{n})_{n\in \mathbb N}\in A^{\beta}$ and $(k_n)_{n\in \N}$ a strictly increasing sequence of natural numbers. Then
$(\xi_{k_n})_{n\in \mathbb N}\in A^{\beta}$.
\end{lemma}

\begin{proof} Take an arbitrary ${\bf x}=(x_n)_{n\in \mathbb N}\in A$ and define ${\bf y}=(y_n)_{n\in \mathbb N}\in X^{\mathbb N}$ as follows: $y_{k_n}=x_n,\,n=1,2,\dots$ and $y_j=0,\,\forall j\in\mathbb N\setminus \{k_1,k_2,\dots\}$. Then ${\bf y}=(y_n)_{n\in \mathbb N}\in A$ by (\ref{6oct1}).\\
Since   ${\bf y}=(y_n)_{n\in \mathbb N}\in A$ and $(\xi_n)_{n\in \mathbb N}\in A^{\beta}$, the sequence
$$
 \left( \prod_{k=1}^n\xi_k(y_k) \right)_{n\in \mathbb N}
 $$
 converges in $\mathbb S$. Observe that
 \begin{equation}\label{eqag1}
 \prod_{j=1}^{k_n}\xi_j(y_j)=\prod_{j=1}^{n}\xi_{k_j}(x_j),\,\,n=1,2,\dots
 \end{equation}
 From  (\ref{eqag1}) we conclude that the sequence
 $$
 \left( \prod_{j=1}^{n}\xi_{k_j}(x_j)\right)_{n\in \mathbb N}
 $$
 converges in $\mathbb S$ too. Since this is true for an arbitrary  ${\bf x}=(x_n)_{n\in \mathbb N}\in A$, we get  that $(\xi_{k_n})_{n\in \mathbb N}\in A^{\beta}$.
\end{proof}

 For a topological abelian group  $X$,  a subset $A\subset X^{\mathbb N}$   and for   a fixed ${\bf h}=(\xi_n)_{n\in \mathbb N} \in  A^{\beta}$ we define a mapping $\chi_{\bf h}:A\to \mathbb S$ by the equality:
 $$
 \chi_{\bf h}({\bf x})=\prod_{n=1}^{\infty}\xi_n(x_n):=\lim_n\prod_{k=1}^{n}\xi_k(x_k),\,\,{\bf x}=(x_n)_{n\in \mathbb N}\in A\,.
 $$
 It is easy to observe that
 \begin{itemize}
 \item if ${\bf h}=(\xi_n)_{n\in \mathbb N}\in (X^{\wedge})^{(\mathbb N)}$, then $\chi_{\bf h}$ is defined on the whole $X^{\N}$.
 \item If $A$ is a subgroup of $X^{\mathbb N}$, then
$$
 \chi_{\bf h}\in Hom\left(A,\mathbb S\right),\quad \forall {\bf h}\in
 A^{\beta}\,.
 $$
   \end{itemize}
  \begin{notation}{\em
 For a subgroup $A\subset X^{\mathbb N}$, the letter $\chi$ will denote  in the sequel the mapping
$$
\chi:A^{\beta}\to Hom\left(A,\mathbb S\right)
$$
 defined by the equality:
$$
\chi({\bf h})=\chi_{\bf h}\,\quad \forall {\bf h}=(\xi_n)_{n\in \mathbb N} \in  A^{\beta}\,.
$$
Of course, the mapping $\chi$ depends on $A$, but, to simplify notation, we shall not indicate this dependence.
}
    \end{notation}

   \begin{lemma}\label{6apr}
 Let $X$ be a topological abelian group and $A\subset X^{\mathbb N}$ a subgroup such that $X^{(\mathbb N)}\subset A$. Then, the mapping $\chi:A^{\beta}\to Hom\left(A,\mathbb S\right)$ is an injective group homomorphism.
 \end{lemma}

 \begin{proof} It is easy to see that $\chi$ is a group homomorphism. Let ${\bf h}=(\xi_n)_{n\in \mathbb N}\in \ker(\chi) $. Then
  $$
  \chi_{\bf h}({\bf x})=1,\quad \forall {\bf x}=(x_n)_{n\in \mathbb N}\in A.
  $$
  Fix $n\in\mathbb N$ and $x\in X$. As $\nu_n(x)\in X^{(\mathbb N)}\subset A$, we get: $\xi_n(x)=\chi_{\bf h}(\nu_n(x))=1$. Since $x\in X$ is arbitrary,  $\xi_n$ must be the null  character. Therefore, ${\bf h}={ 1}:=(1,1,\dots)$ and $\ker(\chi)=\{{1}\}$.  Hence, $\chi$ is injective.
 \end{proof}

%

 The following example illustrate the usefulness of the non-topological Lemma \ref{6apr1}.

\begin{example}\label{kapl} {\em  It is a well-known fact, that} { if $X$ is a topological abelian group, then
\begin{equation}\label{3octeq1}
 \chi_{\bf h}\in \left(X^{\N},\mathfrak{p} \right)^{\wedge}\quad\forall {\bf h}=(\xi_n)_{n\in \mathbb N}\in (X^{\wedge})^{(\mathbb N)}=(X^{\N})^{\beta}\,
\end{equation}
  and the mapping   $\chi:(X^{\wedge})^{(\mathbb N)}\to \left( X^{\N},\mathfrak{p} \right)^{\wedge}$ is a group isomorphism.}

{\em  Indeed, (\ref{3octeq1}) is easy to verify. The injectivity of $\chi$ derives from Lemma \ref{6apr}. It remains to show that $\chi$ is surjective too.
\par
Write $G=\left( X^{\N},\mathfrak{p} \right)$ and fix an arbitrary $\kappa\in G^{\wedge}$. We need to find ${\bf h}=(\xi_n)_{n\in \mathbb N}\in (X^{\wedge})^{(\mathbb N)}$
 such that $\kappa=\chi_{\bf h}$. For every $n\in \mathbb N$ the  homomorphism $\nu_n:X\to G$ is continuous. Hence $\xi_n:=\kappa\circ \nu_n\in X^{\wedge}$. Let us see that  ${\bf h}:=(\xi_n)_{n\in \mathbb N}$ meets the requirements.

 Fix ${\bf x}=(x_n)_{n\in \mathbb N}\in X^{\N}$. Evidently,   the sequence $(\sum_{k=1}^n \nu_k(x_k))_{n\in \mathbb N}$ converges in $G$ to ${\bf x}=(x_n)_{n\in \mathbb N}$ . Hence,
 \begin{equation}\label{eq29oc}
  \kappa({\bf x})=\lim_n\kappa(\sum_{k=1}^n\nu_k(x_k) )=\lim_n\prod_{k=1}^n\xi_k(x_k)\,.
 \end{equation}
  Since ${\bf x}=(x_n)_{n\in \mathbb N}\in  X^{\N}$ is arbitrary, (\ref{eq29oc}) implies that ${\bf h}=(\xi_n)_{n\in \mathbb N}\in \left(X^{\N}\right)^{\beta}$ and  $\kappa=\chi_{\bf h}$. By Lemma \ref{6apr1},   $\left( X^{\N} \right)^{\beta}=(X^{\wedge})^{(\mathbb  N)}\,.$  Consequently we have found ${\bf h}=(\xi_n)_{n\in \mathbb N}\in (X^{\wedge})^{(\mathbb N)}$ such that $\kappa=\chi_{\bf h}$, and  the surjectivity of $\chi$ is thus proved.
}
\end{example}

 \begin{remark}{\em Let $X$ be a topological abelian group.
\begin{itemize}
    \item[(a)] Example \ref{kapl} asserts only that the group $\left(X^{\N},\mathfrak{p} \right)^{\wedge}$ can be algebraically identified with the group $(X^{\wedge})^{(\mathbb N)}$ by means of the group isomorphism $\chi$. In fact more is known:   $\chi$ is also a homeomorphism  between $\left( X^{\N},\mathfrak{p}
\right)^{\wedge}_{\rm{co}}$ and $\left((X^{\wedge})^{(\mathbb N)},\hat{\mathfrak{b}}\right)$, where $\hat{\mathfrak{b}}$ stands
for the topology  induced from the box product $\left((X^{\wedge}_{\rm{co}})^{\mathbb N},\mathfrak{b}
\right)$ in $(X^{\wedge})^{(\mathbb N)}$.
   \item[(b)] An application of $(a)$ for $X=\Z$ gives that $\left({\S}^{(\N)}, \mathfrak{b}|_{{\S}^{(\N)}}\right)$ is a
complete non-metrizable group. In particular, we get that ${\S}^{(\N)}$ is a {\it closed} subgroup of $\left({\S}^{\N},
\mathfrak{b}\right)$.

   \item[(c)]  It is known also that $\left({X}^{(\N)}, \mathfrak{b}|_{{X}^{(\N)}}\right)^{\wedge}_{\rm{co}}$ is topologically isomorphic with $\left((X^{\wedge}_{\rm{co}})^{\mathbb
N},\mathfrak{p}\right)$.
   \item[(d)] An application of $(c)$ for $X=\S$ gives that the group $\left({\S}^{(\N)}, \mathfrak{b}|_{{\S}^{(\N)}}\right)^{\wedge}$ has
cardinality $\cont$. It follows that $\mathfrak{b}|_{{\S}^{(\N)}}$ {\it is not} a compatible topology for $\left({\S}^{(\N)}, \mathfrak{p}|_{{\S}^{(\N)}}\right)$ (cf. Proposition \ref{Maincorollary}).
\end{itemize}
 }
 \end{remark}
\section{ The topological dual of $(c_0(X), \mathfrak{u}_0)$. Coincidence with the  $\beta$-dual}

In this section we will prove that for a complete
metrizable gruoup $X$, the dual of the topological group $(c_0(X),
\mathfrak{u}_0)$ algebraically coincides with the $\beta$-dual. We
start  calculating the $\beta$-dual for the particular group
$c_0(\S)$, which has interest in itself: in fact, from it we derive
the first example of a metrizable locally quasi-convex group which
is not $\rm{LQC}$-Mackey (See Proposition \ref{Maincorollary}).

The following statement is a bit more delicate than Lemma \ref{6apr1}.

\begin{Pro}\label{bolo2oq1}
For $c_0(\S)$ we have:

\begin{equation}\label{deeq}
c_0(\S)^{\beta}=({\S}^{\wedge})^{(\mathbb N)}\,.
\end{equation}
\end{Pro}

\begin{proof}
For a fixed $m\in\mathbb Z$ let $\varphi_m:\mathbb S\to \mathbb S$ be the mapping $t\mapsto t^m$. It is known that
    $$
    \mathbb S^{\wedge}=\{\varphi_m:m\in \mathbb Z\}.
    $$
    So, fix a sequence $(m_n)_{n\in\N}\in\Z^{\N}$ such that $(\varphi_{m_n})_{n\in\N}\in c_0(\S)^{\beta}$ and let us see that in fact  $(m_n)_{n\in\N}\in\Z^{(\N)}$.
  \par
  Suppose that   $(m_n)_{n\in\N}\not\in\Z^{(\N)}$. Then for some strictly increasing sequence $(k_n)_{n\in\N}$ of natural numbers we shall have: $m_{k_n}\ne 0,\,n=1,2,\dots$ As $(\varphi_{m_n})_{n\in\N}\in c_0(\S)^{\beta}$, by Lemma \ref{4oct1} we have:
   \begin{equation}\label{2octeq2}
  (\varphi_{m_{k_n}})_{n\in\N}\in c_0(\S)^{\beta}\,.
  \end{equation}
              Let $x_1=x_2=1$; then for a natural number $j>2$ find the unique natural number $n$ with $2^n<j\le 2^{n+1}$ and write
      $$
      x_j=\exp{\left(2{\pi}i\frac{1}{m_{k_j}2^{n+1}}\right)}\,.
      $$
Clearly, ${\bf x}=(x_j)_{j\in \mathbb N}\in c_0(\mathbb S)$ and
      \begin{equation}\label{2octeq3}
 \prod_{j={2^n}+1}^{2^{n+1}}\varphi_{m_{k_j}}(x_j)=\exp{\left(2{\pi}i\sum_{k={2^n}+1}^{2^{n+1}}\frac{1}{2^{n+1}}\right)}=-1,\,n=1,2,\dots
            \end{equation}
     It follows from (\ref{2octeq3}) that
       $$
       \left(\prod_{j=1}^n\varphi_{m_{k_j}}(x_{j})\right)_{n\in \mathbb N}
       $$
  is not a Cauchy sequence in $\mathbb S$, hence it is not convergent in $\mathbb S$ in contradiction with (\ref{2octeq2}).
          \end{proof}
The next proposition plays a pivotal role in the computation of the dual of  $(c_0(X), \mathfrak{u}_0)$.

\begin{Pro}\label{elenaN1} Let $X$ be a topological abelian group and  $G:=(c_0(X), \mathfrak{u}_0)$. The following assertions hold:
\begin{itemize}
    \item[(a)]   $\chi_{\bf h}\in G^{\wedge}$  for every ${\bf h}=(\xi_n)_{n\in \mathbb N}\in (X^{\wedge})^{(\mathbb N)}\subset c_0(X)^{\beta}$ and the mapping   $\chi:(X^{\wedge})^{(\mathbb N)}\to G^{\wedge}$ is an injective group homomorphism.
    \item[(b)]   We have:
$$
G^{\wedge}\subset \{\chi_{\bf h}: {\bf h}=(\xi_n)_{n\in \mathbb N}\in c_0(X)^\beta\}\,.
$$
    \item[(c)]  If $X$ is complete metrizable (or, more generally, $(c_0(X), \mathfrak{u}_0)$ is a Baire space), then  we have:
     $$
     G^{\wedge}=\{\chi_{\bf h}: {\bf h}=(\xi_n)_{n\in \mathbb N}\in c_0(X)^\beta\}\,
     $$
  and the mapping  $\chi:c_0(X)^\beta\to G^{\wedge}$ is a group isomorphism.
  \end{itemize}
  \end{Pro}

  \begin{proof}  $(a)$ As $\mathfrak{u}_0\ge {\mathfrak{p}}_0$ we have
   $\chi_{\bf h}\in G^{\wedge},\,\forall {\bf h}=(\xi_n)_{n\in \mathbb N}\in (X^{\wedge})^{(\mathbb N)}$.
 The rest follows from 
 Lemma \ref{6apr}.

    $(b)$ Fix $\kappa\in G^{\wedge}$. We need to find ${\bf h}=(\xi_n)_{n\in \mathbb N}\in c_0(X)^\beta$ such that $\kappa=\chi_{\bf h}$. For every $n\in \mathbb N$ the homomorphism $\nu_n:X\to G$ is continuous, so $\xi_n:=\kappa\circ \nu_n\in X^{\wedge}$. Let us see that   ${\bf h}:=(\xi_n)_{n\in \mathbb N}$ meets the requirements.

 Fix ${\bf x}=(x_n)_{n\in \mathbb N}\in c_0(X)$. By Lemma \ref{agu1}  the sequence $(\sum_{k=1}^n \nu_k(x_k))_{n\in \mathbb N}$ converges in $G$
 to ${\bf x}=(x_n)_{n\in \mathbb N}$. Hence,
 \begin{equation}\label{eq29}
  \kappa({\bf x})=\lim_n\kappa(\sum_{k=1}^n\nu_k(x_k))=\lim_n\prod_{k=1}^n\xi_k(x_k)\,.
 \end{equation}
  Since ${\bf x}=(x_n)_{n\in \mathbb N}\in c_0(X)$ is arbitrary, (\ref{eq29}) implies that ${\bf h}=(\xi_n)_{n\in \mathbb N}\in c_0(X)^\beta$ and  $\kappa=\chi_{\bf h}$.
  \par
    $(c)$ Taking into account $(b)$, we only need to see  that
    $$
  G^{\wedge}\supset\{\chi_{\bf h}: {\bf h}=(\xi_n)_{n\in \mathbb N}\in c_0(X)^\beta\}\,.
 $$
 So, fix ${\bf h}=(\xi_n)_{n\in \mathbb N}\in c_0(X)^\beta$. As we have already noted, $\chi_{\bf h}:c_0(X)\to \mathbb S$
 is a group homomorphism. For $n\in \mathbb N$, set ${\bf h}_n=(\xi_1,\dots,\xi_n,1,1,\dots)$. Then ${\bf h}_n \in (X^{\wedge})^{(\mathbb N)}$. Hence,
 \begin{equation}\label{3oct1}
 \chi_{{\bf h}_n}\in G^{\wedge},\,\,n=1,2,\dots
\end{equation}
 Clearly,
 \begin{equation}\label{5apr}
 \lim_n\chi_{{\bf h}_n}({\bf x})=\chi_{\bf h}({\bf x}),\,\quad \forall {\bf x}=(x_n)_{n\in \mathbb N}\in c_0(X)\,.
  \end{equation}
Since $X$ is complete metrizable, by Proposition \ref{descr2m}$(c)$, the group  $G=(c_0(X), \mathfrak{u}_0)$ is complete metrizable too. In particular, $G=(c_0(X), \mathfrak{u}_0)$
is a Baire space. This and relations (\ref{3oct1}) and (\ref{5apr}), according to Osgood's theorem \cite[Theorem 9.5 (pp. 86--87)]{KeNa}
imply  that the function $\chi_{\bf h}$ has a $\mathfrak{u}_0$-continuity point   ${\bf x}=(x_n)_{n\in \mathbb N}\in c_0(X)$.
  Since   $\chi_{\bf h}$ is a group homomorphism, we get that  $\chi_{\bf h}$ is $\mathfrak{u}_0$-continuous. Therefore, $\chi_{\bf h}\in G^{\wedge} $.
  \end{proof}
Now we are ready to give the first example of a precompact metrizable group which is not a Mackey group in $\rm{LQC}$:

  \begin{Pro}\label{Maincorollary}   For $X = \S $ the following assertions hold:

   \begin{itemize}
   \item[(a)] {\em (cf. \cite[Lemma]{Ni})} $(c_0(\S), \mathfrak{p}_0)^{\wedge}=\left({c_0(\S)}, \mathfrak{u}_0\right)^{\wedge}
   $. In particular, the set $\left({c_0(\S)}, \mathfrak{u}_0\right)^{\wedge}$ is
   countable.

   \item[(b)] $\mathfrak{u}_0$ is a compatible locally quasi-convex Polish group topology for
   $\left({c_0(\S)},  \mathfrak{p}_0\right)$.
  \item[(c)]  $\left({c_0(\S)},  \mathfrak{p}_0\right)$ is a precompact metrizable group which is not a Mackey group in $\rm{LQC}$.
  Further, it is  connected and monothetic.
   \end{itemize}
   \end{Pro}

   \begin{proof}
%
  $(a)$ Since $\mathfrak{p}_0\le \mathfrak{u}$,  we have: $\left({c_0(\S)}, \mathfrak{p}_0\right)^{\wedge}\subset\left({c_0(\S)},
  \mathfrak{u}_0\right)^{\wedge}$. To prove the converse  inclusion, fix an arbitrary $\kappa\in \left({c_0(\S)},  \mathfrak{u}_0\right)^{\wedge}$. By Proposition
  \ref{elenaN1}$(b)$ there exists ${\bf h}=(\xi_n)_{n\in \mathbb N}\in \left(c_0(\S)\right)^{\beta}$
  such that $\kappa=\chi_{\bf h}$. By Proposition \ref{deeq}
$c_0(\S)^{\beta}=({\S}^{\wedge})^{(\mathbb N)}\,.$ Therefore we
have: $\kappa=\chi_{\bf h}$,  where ${\bf h}=(\xi_n)_{n\in \mathbb
N}\in ({\S}^{\wedge})^{(\mathbb N)}\,.$ Consequently, $\kappa\in
\left({c_0(\S)}, \mathfrak{p}_0\right)^{\wedge}$ and \NB the first
part of  $(a)$ is proved. \NB The second part of $(a)$ follows from
the first one because $(c_0(\S), \mathfrak{p}_0)^{\wedge}$ is
algebraically isomorphic to $(\S^{\N}, \mathfrak{p})^{\wedge}$.

$(b)$ $\mathfrak{u}_0$ is a  locally quasi-convex Polish group topology by Proposition \ref{descr2m}.
 By $(a)$ it is compatible for $\left({c_0(\S)},  \mathfrak{p}_0\right)$.

$(c)$ Observe that $\left({c_0(\S)}, \mathfrak{p}_0\right)$ is a
topological subgroup
of  the compact metrizable group $({\S}^{\N},\mathfrak{p})$. Therefore it is metrizable and precompact.
By  $(b)$, $\mathfrak{u}_0$ is a locally quasi-convex group topology compatible for $\left({c_0(\S)}, \mathfrak{p}_0\right)$ and
strictly finer than $\mathfrak{p}_0$ (by  Proposition \ref{descr2m}).
  This proves that $\left({c_0(\S)}, \mathfrak{p}_0\right)$ is not a Mackey group in $\rm{LQC}$.

The last two assertions follow respectively from  Proposition \ref{descr2m} $(c)$  and from Remark \ref{Ro}.
  \end{proof}
  \begin{remark}{\em
  It follows from Proposition \ref{descr2m} and Proposition
  \ref{Maincorollary}$(a)$, that $\left({c_0(\S)},
  \mathfrak{u}_0\right)$ is a non-precompact locally quasi-convex
  group with countable dual $\left({c_0(\S)},
  \mathfrak{u}_0\right)^{\wedge}$; consequently by Theorem \ref{elena2}
  $\left({c_0(\S)},
  (\mathfrak{u}_0)^+\right)$ is a metrizable precompact group which
  is not a Mackey group in  $\rm{LQC}$. Observe, however that (again
  by Proposition
  \ref{Maincorollary}$(a)$) we have:
  $(\mathfrak{u}_0)^+=\mathfrak{p}_0$ and  we get a second
  proof of Proposition  \ref{Maincorollary}$(c)$.}
  \end{remark}

  We shall see below that  the group  $\S$ in  Proposition \ref{Maincorollary} can be replaced by an arbitrary non-trivial compact connected metrizable  group (see Theorem \ref{basth}). However the proof of this fact will require a subtle preparation, to which the rest of the paper is devoted.

   \section{The class $\(SB)$}

In this section we introduce a large class of compact metrizable groups $X$ that can be used as input in Proposition  \ref{Maincorollary}.

    \begin{definition}\label{SB}
    {\em For a topological abelian group  $X$, let
     $$
     \Gamma_{\rm{abs}}(X):=\{\xi\in  X^{\wedge}:(\xi(x_n))_{n\in \mathbb N}\in \rm{ss}(\S)\quad\forall {\bf x}=(x_n)_{n\in \mathbb N}\in c_0(X) \}\,.
       $$
         Clearly, $\Gamma_{\rm{abs}}(X)$ is a subgroup of $ X^{\wedge}$.
Denote by  $\(SB)$  the class of groups $X$ 
 such that $\Gamma_{\rm{abs}}(X)=\{1\}$.}
       \end{definition}

       \begin{remark}

        {\em Let $X$ be a topological abelian group.

\begin{itemize}
     \item[(a)]
      The notation $\Gamma_{\rm{abs}}(X)$ is justified by  the following facts easy to prove:
       $$
     \Gamma_{\rm{abs}}(X)=\{\xi\in  X^{\wedge}:(\xi(x_n))_{n\in \mathbb N}\in l(\S), \quad\forall {\bf x}=(x_n)_{n\in \mathbb N}\in c_0(X) \}\,.
       $$

   Taking into account the equality (\ref{rd2}), for any character $\xi\in  X^{\wedge}$, we have:
 $$\xi \in\Gamma_{\rm{abs}}(X) \mbox{  iff   } \sum_{k=1}^{\infty}|1-\xi(x_k)|<\infty, \, \forall  {\bf x}=(x_n)_{n\in \mathbb N}\in c_0(X).$$

     \item[(b)] $\Gamma_{\rm{abs}}(X)$  can be described also by means of the diagonal homomorphism $\Delta:  X^{\wedge} \longrightarrow (X^{\wedge})^\N$. Clearly, if  $\xi\in X^{\wedge}$, then  $\xi\in \Gamma_{\rm{abs}}(X)$  iff $(\xi,\xi,\dots,\xi,\dots)\in c_0(X)^\beta$. Therefore:
$$
\Gamma_{\rm{abs}}(X) = \Delta^{-1}(c_0(X)^\beta)
.$$
Whenever  $c_0(X)^\beta = (X^{\wedge})^{(\N)}$,  $\Gamma_{\rm{abs}}(X) =\{1\}$, and
    $X\in \(SB)$. In particular, by Proposition \ref{bolo2oq1}, $\S\in \(SB)$.
             \item[(c)]  We shall show below (Theorem \ref{NewTh}), that if  $X\in \(SB)$ is a  precompact group , then $(c_0(X), \mathfrak u_0)^\wedge$ can be canonically identified with $(X^{\wedge})^{(\N)}$.
             Thus, summarizing: \\

  \centerline{\em $c_0(X)^\beta = (X^{\wedge})^{(\N)} \Longrightarrow X\in \(SB), \;\;\;\;\;\;\;\;$ whilst  $\;\;\;\; \;\;\;\;X\in \(SB) \; \&\;  X$ is precompact $ \Longrightarrow c_0(X)^\wedge = (X^{\wedge})^{(\N)}$.
  }

           \medskip         \medskip

 Our ultimate aim to have the equality $c_0(X)^\wedge = (X^{\wedge})^{(\N)}$ has motivated the  introduction of the class $\(SB)$.
    \end{itemize}}
       \end{remark}

\NB In what follows assume that  $c_0(X)$ is endowed with the topology $\mathfrak u_0$. The   dual $c_0(X)^\wedge $ is a subgroup of the group $Hom(c_0(X), \S)$, and  the density of $X^{(\N)}$ in   $ c_0(X)$ allows the algebraic identification of
 $c_0(X)^\wedge $  with a subgroup of $(X^{\wedge})^{\N}$ contained in $c_0(X)^\beta$ (see also the assignment $\kappa \mapsto (\xi_n)$ from the proof of item (b) of Proposition \ref{elenaN1}). In the sequel we denote by $ j: G^\wedge\to (X^{\wedge})^{\N}$ this assignment $\kappa \mapsto (\xi_n)$, which actually is the inverse of $\chi$ defined in the previous section.
  Then $j(G^\wedge) \subseteq c_0(X)^\beta$.

 \begin{theorem}\label{NewTh}
 If  $X\in\(SB)$ is precompact, and $G:=(c_0(X), \mathfrak{u}_0)$, then $j(G^\wedge)= (X^{\wedge})^{(\N)}$.  Hence the mapping  $\chi:(X^{\wedge})^{(\mathbb N)}\to G^{\wedge}$ is a  group isomorphism.
    \end{theorem}

\begin{proof} By Proposition \ref{elenaN1}$(a)$, we only need to see that $\chi$ is surjective. To this end fix
$$
{\bf h}=(\xi_n)_{n\in \mathbb N}\in j(G^\wedge) \subseteq c_0(X)^\beta \subseteq (X^{\wedge})^{\mathbb N}.
$$
 We have to see that, in fact, ${\bf h}\in (X^{\wedge})^{(\mathbb N)}$.\\
 {\it By using the $\mathfrak{u}_0$ continuity of $ \chi_{\bf h}$}  we can find  some $V\in \mathcal N(X)$ such that
$\chi_{\bf h}\left( V^{\N} \cap c_0(X)\right)\subset \S_+\,$. This implies:
 \begin{equation}\label{imcon}
 \xi_n\in  V^{\triangleright},\quad \forall n\in \mathbb N\,.
 \end{equation}
  By Proposition \ref{2oct} {\it the precompactness of  $X$ implies that  the set $V^{\triangleright}$ is finite}. Thus, by (\ref{imcon}), the set
 $$
W_{{\bf h}}:=\{\xi\in V^{\triangleright}:\exists n\in \mathbb N,\,\,\xi=\xi_n\}
 $$
 is finite. Fix $\xi\in W_{{\bf h}}$ and write
 $$
\mathbb N_{\xi}:=\{n\in \mathbb N: \xi=\xi_n\}.
 $$
 Clearly, $\mathbb N_{\xi}\ne \emptyset,\,\,\forall \xi\in W_{{\bf h}}$.  We need to prove the following\\

{\bf Claim.} {\em If  the set $\mathbb N_{\xi}$ {\it is infinite} for some $\xi\in W_{{\bf h}}$, then $\xi=1$.} \\

{\sl Proof of the Claim.} Let  $\mathbb N_{\xi}$ {\it is infinite} for
$\xi\in W_{{\bf h}}$. Write $\mathbb N_{\xi}$ as a strictly increasing sequence: $\mathbb N_{\xi}=\{k_1,k_2,\dots\}$.

From $(\xi_n)_{n\in \mathbb N}\in c_0(X)^\beta$ by Lemma \ref{4oct1} we conclude that  $(\xi_{k_n})_{n\in \mathbb N}\in c_0(X)^\beta$. Hence, $(\xi,\xi,\dots,\xi,\dots)\in c_0(X)^\beta$, and so,  $\xi\in \Gamma_{\rm{abs}}(X)$. As $X\in \(SB)$, we obtain that $\xi=1$. This proves the claim. \\

As the set $ W_{{\bf h}}$ {\bf is finite}, we deduce from the claim that for some $n_0\in \mathbb  N$ we have:  $\xi_n=1,\,\forall n\ge n_0$. Consequently,  ${\bf h}=(\xi_n)_{n\in \mathbb N}\in (X^{\wedge})^{(\mathbb N)}$ and the surjectivity of $\chi$ is proved.
\end{proof}

\begin{Cor}\label{NewNew} If  $X\in \(SB)$ is a non-trivial precompact group and  $G:=(c_0(X), \mathfrak{u}_0)$, then
 $\mbox{Card } G^{\wedge}=\mbox{Card } X^{\wedge}$.
\end{Cor}

\begin{proof}
According to Theorem \ref{NewTh}, the mapping  $\chi:(X^{\wedge})^{(\mathbb N)}\to G^{\wedge}$ is a  group isomorphism.
 Hence, $\mbox{Card } G^{\wedge}=\mbox{Card } (X^{\wedge})^{(\mathbb N)}$.
 Clearly, for  any precompact nontrivial group $X \in \(SB)$, $X^\wedge$ is infinite. Otherwise $X$ would also be finite  and   $c_0(X)$ would coincide with $X^{(\N)}$  Then $\Gamma_{\rm{abs}}(X) \neq 1$, which contradicts the fact that  $X \in \(SB)$.
Thus,   $\mbox{Card }G^{\wedge}=\mbox{Card } (X^{\wedge})^{(\mathbb N)}= \mbox{Card } X^{\wedge}$.
\end{proof}

\subsection{Properties of the class $\(SB)$}

  It is clear that $\(SB)$ contains all minimally almost periodic groups. We will consider soon
other  interesting examples.
  \begin{Pro}\label{tirol} Let $X$ be a topological abelian group.
\begin{itemize}
     \item[(a)]  If $\rm{cs}(X)= c_0(X)$, then $\Gamma_{\rm{abs}}(X)=X^{\wedge}$.
     \item[(b)] If $\rm{cs}(X)=c_0(X)$ and $X^{\wedge}\ne\{1\}$, then $X\not \in \(SB)$.
     \item[(c)] If $X\ne \{0\}$ is locally compact and totally disconnected, then $X\not \in \(SB)$.
\end{itemize}
    \end{Pro}

\begin{proof} $(a)$ is easy to verify and $(b)$ follows from $(a)$. Finally, $(c)$ follows from $(b)$ and Proposition \ref{alonso} $(b)$ because $X^{\wedge}\ne\{1\}$.
\end{proof}

The class $\(SB)$ is stable through  continuous homomorphisms (and in particular, through quotients), as proved in the next lemma..

\begin{lemma}\label{wild1} Let $X,Y$ be topological abelian groups and  $\varphi\in CHom(X,Y)$. We have:
\begin{itemize}
    \item[(a)] $\varphi^{\wedge}(\Gamma_{\rm{abs}}(Y))\subset\Gamma_{\rm{abs}}(X)$.
    \item[(b)] If $X\in \(SB)$ and $\varphi(X)$ is dense in $Y$, then $Y\in \(SB)$.
    \item[(c)] If ${\mathbf n}(X)$ is the  von-Neumann's kernel of $X$, then  $X\in \(SB)$ iff $X/{\mathbf n}(X)\in \(SB)$.
\end{itemize}
\end{lemma}

\begin{proof} $(a)$ is easy to verify.\\
$(b)$ By  the  assumption $\Gamma_{\rm{abs}}(X)=\{1\}$, and from $(a)$ we get that $\varphi^{\wedge}(\Gamma_{\rm{abs}}(Y))=\{1\}$. Now it remains to note that $\varphi^{\wedge}$ is injective by the density of $\varphi(X)$ in $Y$.

(c) The implication $X\in \(SB) \Longrightarrow X/{\mathbf n}(X)\in \(SB)$ follows from (b). The implication $X/{\mathbf n}(X)\in \(SB) \Longrightarrow X\in \(SB)$
follows from the fact that  the canonical homomorphism $\varphi : X\to X/{\mathbf n}(X)$ induces an isomorphism $\varphi^\wedge: (X/{\mathbf n}(X))^\wedge
\to X ^\wedge$.
\end{proof}

      Now we prove that the class $\(SB)$ is stable also under arbitrary direct products.

    \begin{Pro}\label{DikLemma2} Let $I$ be a non-empty index set,  $(X_i)_{i\in I}$ a family of topological groups. Then the cartesian product $\prod_{i\in I}X_i$ belongs to
$\(SB)$ iff $X_i\in \(SB)$ for every $i\in I$.
  \end{Pro}

  \begin{proof} Assume that $X_i\in \(SB)$ for every $i\in I$. Fix $\varphi\in (\prod_{i\in I}X_i)^{\wedge}$. It is known (see, e.g., \cite{B} or \cite[Exercise 2.10.4(g,h)]{DPS}) that  there is a family $(\xi_i)_{i\in I}\in \prod_{i\in I}X_i^{\wedge}$ such that  ${\rm{Card}}\{i\in I:\xi_i\ne 1\}<\infty$ and
  $$
  \varphi({\bf x})=\prod_{i\in I}\xi_i(x_i),\quad \forall {\bf x}=(x_i)_{i\in I}\in \prod_{i\in I}X_i\,.
  $$
  Suppose now that $\varphi\in \Gamma_{\rm{abs}}(\prod_{i\in I}X_i)$. It is easy to see that  $\xi_i\in \Gamma_{\rm{abs}}(X_i),\,\forall i\in I$.   From the assumption
  $\Gamma_{\rm{abs}}(X_i)=\{1\},\,\forall i\in I$, we get that $\xi_i=1,\,\forall i\in I$. Hence, $\varphi=1$ and so, $\Gamma_{\rm{abs}}(\prod_{i\in I}X_i)=\{1\}$.
  \par
\NB The converse follows from Lemma \ref{wild1}, (b).
\end{proof}

  \subsection{A description of the class $\(SB)$}

Here we offer a complete description of the metrizable groups in $\(SB)$ (see Theorem \ref{dikconj} and Corollary \ref{Coro_dikconj}).  The following notion due to  Enflo  \cite{Enf} is a cornerstone in the sequel.

\begin{definition}\label{Def_loc_gen}{\em \cite[p. 236]{Enf} A topological group $X$ is called {\it locally generated} if
$$
\langle V\rangle=X\quad \forall V\in \mathcal N(X)\,.
$$
}\end{definition}

It is easy to observe that a topological group $X$ is locally generated iff
$$
X = \bigcup_{k= 1}^{\infty}  (V + \cdots^{k\,\text{summands}} + V),\quad  \forall V\in \mathcal N(X)\,.
$$

 Some easy properties and known facts of the  locally generated groups are collected in the next Remark.

\begin{remark}{\em
\begin{itemize}
     \item[(a)] All connected topological groups are locally generated  {\rm (\cite[Theorem II.7.4 ]{HR})}.  On the other hand,
every locally generated {\it locally compact} group  is connected \cite[Corollary II.7.9 ]{HR}.
      \item[(b)]  Obviously, a group $X$ is  locally generated iff $X$ has no proper open subgroups. Consequently, if
$H$ is a dense subgroup of a topological group $G$, then $H$ is locally generated iff $G$ is locally generated.
      \item[(c)]  From (a) and (b) one can deduce that a locally precompact group is locally generated iff its completion is connected (i.e., the locally generated locally precompact groups are precisely the dense subgroups of the connected locally compact groups).
      \item[(d)]  The additive group of rational numbers $\mathbb Q$ with the usual topology is a  metrizable locally generated group which is totally disconnected. A complete metrizable  locally generated topological abelian group also can be totally disconnected \cite[Example 2.2.1]{Enf}.
\end{itemize}}
\end{remark}

       \begin{theorem}\label{dikconj}  Let $X$ be a topological abelian group.
\begin{itemize}
     \item[(a)] If $X\in \(SB)$, then $X$ is locally generated.
     \item[(b)] If  $X$ is  locally generated and  metrizable, then $X\in \(SB)$.
\end{itemize}
         \end{theorem}

 \begin{proof} $(a)$ Take a symmetric open $V\in \mathcal N(X)$ and let $H:=\langle V\rangle$.  Then $H$ is open in $X$ and so, the quotient group $X/H$ is discrete. Now from
Lemma \ref{wild1} $(b)$ we get that the discrete group  $X/H\in \(SB)$. This implies that $X/H$ is a singleton. Consequently,       $H=X$ and so, $X$ is locally generated.
       \par
       $(b)$ Take a character $\xi \in X^\wedge \setminus\{1\}$. In order to prove that $\xi \notin \Gamma_{abs}(X)$ we must find a sequence $(x_n)_{n\in \mathbb N}\in c_0(X)$ such that
       $$
       \left(\prod_{k=1}^n\xi (x_k)\right)_{n\in \mathbb N}
       $$
      is not  convergent in $\mathbb S$.
\par
As  $\xi \in X^\wedge \setminus\{1\}$, there is $x\in X$ such that $\xi(x)\ne 1$. Let  $\{V_1, V_2,... \}$ be a basis for $\mathcal N(X)$ such that $V_n\supset V_{n+1},n=1,2,\dots$.  Since $X$ is locally  generated,
$$
X = \bigcup_{k= 1}^{\infty}  (V_n + \cdots^{k\,\text{summands}} + V_n),\quad \forall n\in \N\,.
$$
       Thus for a given $n\in \N$ we can find $k_n\in \N$ such that
$$
x\in V_n + \cdots^{k_n\,\text{summands}} + V_n\,.
$$
Therefore, 
we can also find a finite sequence  $x_{n, 1},\dots, x_{n, k_n}$ such that
$$
x_{n, i} \in V_n,\,i=1,\dots,k_n
$$
and
$$
x=\sum_{i=1}^{k_n}x_{n, i}\,.
$$
Let
$$
m_0=0,\,\,m_n:=\sum _{i=1}^{n}k_i,\quad n=1,2,\dots
$$
Define now a sequence $(x_j)_{j\in \N}$ as follows: find for $j\in \N$ the unique $n\in \N$ with $m_{n-1}<j\le m_n$ and put
$$
x_j:=x_{n,j-m_{n-1}}\,.
$$
Clearly,
$$
x_1=x_{1,1},\dots,x_{m_1}=x_{1,m_1},x_{m_1+1}=x_{2,1},\dots,x_{m_2}=x_{2,k_2},x_{m_2+1}=x_{3,1},\dots,x_{m_3}=x_{3,k_3},\dots
$$
and

$$
x_j\in V_n,\,j=m_{n-1}+1,\dots, m_{n};\,\,n=1,2,3,\dots
$$
The last relation, since $m_n\to \infty$ and $(V_n)_{n\in \N}$ is a decreasing basis for $\mathcal N(X)$, implies that the sequence $(x_j)_{j\in \N}$ converges to zero in $X$. Now,
$$
 \prod_{j=m_{n-1}+1}^{m_{n}} \xi(x_j)=\prod_{i=1}^{k_n} \xi(x_{n,i}) = \xi\left(\sum _{i=1}^{k_n}x_{n,i}\right) = \xi(x)\ne 1,\,\,n=1,2,3,\dots$$

Consequently, $\left(\prod_{j=1}^{n}\xi(x_j)\right)_{n\in \mathbb N}$ is not a Cauchy sequence in  $\S$
  and hence  it is not convergent.
\end{proof}

Since connected groups are locally generated, we obtain:

\begin{corollary}\label{Coro_dikconj} A metrizable abelian group $X\in \(SB)$ iff $X$ is locally generated.  In particular, $\(SB)$ contains all connected metrizable groups.
            \end{corollary}

       \begin{remark}{\em  Since $\(SB)$ contains all minimally almost periodic groups  (see Lemma \ref{wild1} (c)), from  Theorem \ref{dikconj}$(a)$ we conclude that a minimally almost periodic group is necessarily locally generated. In \cite[p. 21]{DPS} this observation is used for producing a Hausdorff group topology $\tau$ in $\Z^{(\N)}$ such that $(\Z^{(\N)},\tau)$ is minimally almost periodic.}
\end{remark}

The following theorem  implies, in particular, that the metrizability assumption can be removed from Theorem \ref{dikconj} $(b)$ in the locally compact case (however this cannot be done in general, see Remark \ref{5oct})).

\begin{theorem}\label{dikTh} For a locally compact abelian group $X$ TFAE:
\begin{itemize}
    \item[(i)] $X\in \(SB)$.
    \item[(ii)] $X$ is locally generated.
    \item[(iii)] $X$ is connected.
\end{itemize}
\end{theorem}

\begin{proof}
$(i)\Longrightarrow (ii)$. By Theorem \ref{dikconj} $(a)$, (i) implies that $X$ is locally generated.
\par
$(ii)\Longrightarrow (iii)$. Follows from  \cite[Corollary (7.9)]{HR}.

$(iii)\Longrightarrow (i)$ Take $ \xi\in \Gamma_{\rm{abs}}(X)$ and let us verify that $\xi=1$.
\par
Consider the set
$$
A=\bigcup_{\varphi\in CHom(\R,X)}\varphi(\R)\,.
$$
Let us see first that
\begin{equation}\label{25seq}
\xi|_A=1\,.
\end{equation}
Fix $\varphi\in CHom(\R,X)$ and set $H=\varphi(\R)$. Clearly $\xi|_H \in \Gamma_{\rm{abs}}(H)$. By Theorem \ref{dikconj} $(b)$, $\R\in \(SB)$.   By Lemma \ref{wild1}$(b)$,  $H=\varphi(\R)\in \(SB)$ too. Therefore, $\xi|_H \in \Gamma_{\rm{abs}}(H)=\{1\}$ and $\xi|_H=1$. Consequently (\ref{25seq}) is proved.

 Clearly (\ref{25seq}) implies:
\begin{equation}\label{25seq2}
\xi|_{\langle A\rangle}=1\,.
\end{equation}
Now, according to \cite[Theorem 25.20 (p.410)]{HR} the connectedness of $X$ implies that $\langle A\rangle$ is a {\it dense} subgroup of $X$. From this and (\ref{25seq2}) we obtain that $\xi=1$. \end{proof}

\begin{remark}\label{5oct} {\em Local compactness is essential for the implication $(iii)\Longrightarrow (i)$ of Theorem \ref{dikTh}. Indeed, that
implication may fail in general even for a pseudocompact group $X$. In fact, according to \cite[Corollary 2.10]{GaGa}  (see also  \cite[Remark 3.4]{Tk}) there exists an infinite connected pseudocompact abelian group $X$ which contains no non-trivial convergent sequence. For such a group we have that $c_0(X)=X^{(\N)}$; hence $\Gamma_{\rm{abs}}(X)=X^{\wedge}$ and so, $X\not\in \(SB)$. Consistent examples of connected countably compact groups without infinite compact subsets can be found in \cite[Corollary 2.21]{DikSh} (note that in both cases the groups are
 sequentially complete).
}
\end{remark}
\begin{remark}\label{bourEr}{\em From the proof of Theorem \ref{dikconj}$(b)$ and implication $(ii)\Longrightarrow (i)$ of Theorem \ref{dikTh} it follows that {\it if  a locally generated Hausdorff topological abelian group $X$ is either complete metrizable or locally compact, then $c_0(X)\ne ss(X)$.} According to  \cite[Ch. III. 5, Exercise 6 (b)(p. 315)]{{BourI}} if $X\ne\{0\}$ is  an arbitrary locally generated  {\it complete} Hausdorff topological abelian group, then $ l(X) \ne  \rm{ss}(X)$ and hence, $c_0(X)\ne \rm{ss}(X)$ as well. Remark \ref{5oct} implies that
similar conclusion may fail for a connected (hence locally generated)  {\it sequentially complete} Hausdorf pseudocompact abelian group. }
\end{remark}

Another class  of non-metrizable groups in $\(SB)$ can be obtained from Proposition \ref{DikLemma2}.


   \section{Applications of the class $\(SB)$}

\subsection{Groups with countable dual}

 \begin{Pro}\label{elenaN}
 Let $X\ne \{0\}$ be a compact abelian group and $G:=(c_0(X), \mathfrak{u}_0)$. We have:
  \begin{itemize}
    \item[(a)] If $X$ is connected, then $\mbox{{\rm Card} }G^{\wedge}=\mbox{{\rm Card} } X^{\wedge}$.
    \item[(b)] If $X$ is connected and metrizable, then $ \mbox{{\rm Card} }G^{\wedge}=\aleph_0$.
    \item[(c)] If  $X$  a metrizable and disconnected,  then $\mbox{{\rm Card} } G^{\wedge}=\cont$.
\end{itemize}
  \end{Pro}

  \begin{proof}  $(a)$ follows from Corollary \ref{NewNew} via implication $(iii)\Longrightarrow (i)$ of Theorem \ref{dikTh}.
 \par
 $(b)$ follows from  Theorem \ref{dikconj} $(b)$,  and the equality $\rm{Card } X^{\wedge} = \aleph_0$ ( Proposition \ref{codu1}).
\par
 $(c)$  It is easy to verify that
\begin{equation}\label{kar}
\xi\in \Gamma_{\rm{abs}}(X)\,\Longrightarrow\, \{1,\xi\}^{\N}\subset
c_0(X)^\beta\,.
\end{equation}
Since $X$ is {\it compact metrizable}, 
 by Proposition \ref{elenaN1} $(c)$ we get:
\begin{equation}\label{kar2}
\xi\in \Gamma_{\rm{abs}}(X)\,, {\bf h}\in \{1,\xi\}^{\N} \Longrightarrow\, \chi_{{\bf h}}\in G^{\wedge}\,.
\end{equation}
 Now by Theorem \ref{dikTh}, $X\not \in \(SB)$  and consequently  there exists $\xi \in \Gamma_{\rm{abs}}(X)$, with $\xi\ne 1$. Then $\mbox{Card } \{1,\xi\}^{\N}=\cont$. Therefore,
$$
\{\chi_{{\bf h}}:{\bf h}\in \{1,\xi\}^{\N}\}\subset G^{\wedge}
$$
and since the correspondence ${\bf h}\mapsto \chi_{{\bf h}}$ is injective, we get that $\mbox{Card } G^{\wedge}\ge \cont$.
\end{proof}

The following statement shows that Proposition \ref{elenaN}$(b)$ is the best possible in the class of locally compact groups.

   \begin{Pro}\label{dikrfor2}
For an infinite  locally compact Hausdorff topological abelian group $X$ TFAE:
\begin{itemize}
     \item[(i)]  $X$ is  compact connected and metrizable.
     \item[(ii)]  $\mbox{{\rm Card} }(c_0(X), \mathfrak{u}_0)^{\wedge}=\aleph_0$.
\end{itemize}
\end{Pro}

\begin{proof} $(i)\Longrightarrow (ii)$ by Proposition \ref{elenaN}$(b)$.\\
$(ii)\Longrightarrow (i)$. Let us see first that $X$ is  compact  and metrizable. Write: $G:=(c_0(X), \mathfrak{u}_0)$. Let $\varphi:G\to X$ be the first projection, i.e., the mapping which sends $(x_n)_{n\in \N}$ to $x_1$. Then $\varphi^{\wedge}:X^{\wedge}\to G^{\wedge}$ is injective. So, $\mbox{{\rm Card} } G^{\wedge}=\aleph_0$ implies that
$\mbox{{\rm Card} } X^{\wedge}\le \aleph_0$. From this by Proposition \ref{codu1} we get that $X$ is compact metrizable. Then  by Proposition \ref{elenaN}$(c)$ the equality $\mbox{{\rm Card} } (c_0(X), \mathfrak{u}_0)^{\wedge} =\aleph_0$ implies that $X$ is connected.
\end{proof}

The connected groups in the class of all compact metrizable abelian groups can be characterized also as follows:

\begin{Pro}\label{dikrfor} For an infinite  compact metrizable abelian group $X$ TFAE:
\begin{itemize}
    \item[(i)]  $X$ is connected.
    \item[(ii)] $(c_0(X),\mathfrak{u}_0)^{\wedge}=(X^{\wedge})^{(\N)}$;
    \item[(iii)]  $\mbox{{\rm Card} } (c_0(X),\mathfrak{u}_0)^{\wedge}=\aleph_0$.
    \item[(iv)] $\mbox{{\rm Card} } (c_0(X),\mathfrak{u}_0)^{\wedge}<\cont$.
\end{itemize}
\end{Pro}

\begin{proof} $(i)\Longrightarrow (ii) \Longrightarrow (iii)$ by Proposition \ref{elenaN}$(a)$. $(iii) \Longrightarrow (iv)$ is clear.\\
$(iv)\Longrightarrow (i)$ by Proposition \ref{elenaN}$(c)$.
\end{proof}

  The following theorem provides a wide class of precompact metrizable   groups which are not  Mackey  in $\rm{LQC}$, extending thus the result of  Proposition \ref{Maincorollary}.

\begin{theorem}\label{basth}
Let $X$ be an infinite  compact  connected metrizable group.  We have:
\begin{itemize}
    \item[(a)] $\left({c_0(X)}, \mathfrak{p}_0\right)^{\wedge}=\left({c_0(X)},  \mathfrak{u}_0\right)^{\wedge}$.
    \item[(b)] $\mathfrak{u}_0$ is a locally quasi-convex Polish group topology compatible for $\left({c_0(X)},  \mathfrak{p}_0\right)$.
    \item[(c)] $\left({c_0(X)},  \mathfrak{p}_0\right)$ is a precompact metrizable   group which is not a Mackey group in $\rm{LQC}$. Further, it is connected.
\end{itemize}
  \end{theorem}

  \begin{proof}  $(a)$ Since $\mathfrak{p}_0\le \mathfrak{u}_0$,  we have: $\left({c_0(X)}, \mathfrak{p}_0\right)^{\wedge}\subset\left({c_0(X)},
  \mathfrak{u}_0\right)^{\wedge}$. To prove the converse  inclusion, fix an arbitrary $\kappa\in \left({c_0(X)},
  \mathfrak{u}_0\right)^{\wedge}$. By Proposition  \ref{dikrfor} there exists ${\bf h}=(\xi_n)_{n\in \mathbb N}\in ({X}^{\wedge})^{(\mathbb N)}$ such that $\kappa=\chi_{\bf h}$. Consequently, $\kappa\in \left({c_0(X)}, \mathfrak{p}_0\right)^{\wedge}$ and $(a)$ is proved.
\par
$(b)$ $\mathfrak{u}_0$ is a  locally quasi-convex Polish group topology by Proposition \ref{descr2m}. By $(a)$ it is
compatible for $\left({c_0(X)},  \mathfrak{p}_0\right)$.
\par
$(c)$ $\left({c_0(X)},  \mathfrak{p}_0\right)$ is a precompact metrizable   group because it is a topological subgroup of  the compact metrizable group $({X}^{\N},\mathfrak{p})$. It is not a Mackey group in $\rm{LQC}$ because by $(b)$ $\mathfrak{u}_0$ is a locally quasi-convex  group topology,  compatible for $\left({c_0(X)},
  \mathfrak{p}_0\right)$ and strictly finer than $\mathfrak{p}_0$  (Proposition \ref{descr1} $(a_1)$ ). Connectedness
  follows from Proposition \ref{descr2m}  $(c)$.
  \end{proof}

\subsection{Groups with uncountable dual}

The following statement shows that the topological group $({\mathbb S}^\mathbb N,\mathfrak{u})$ has a  dual much  "bigger" than $(c_0(\mathbb S),\mathfrak{u}_0)$.

\begin{Pro} Let $X\ne \{0\}$ be a compact  group, $G=({X}^\mathbb N,\mathfrak{u})$. Then ${\rm{Card} }\ CHom (G,X) \ge2^{\cont}$.
\par
In particular, if $G=\S$, then $\rm{Card }\ G^{\wedge}= 2^{\cont}$.
 \end{Pro}

 \begin{proof} Denote by  $\mathfrak F$ the set of all ultrafilters on $\mathbb N$. It is known that
\begin{equation}\label{2sep0}
\rm{Card }\ \mathfrak F =2^{\rm{c}}\, .
\end{equation}
 For a  filter $\mathcal F$ on $\mathbb N$, $(x_n)_{n\in \mathbb N}\in {X}^\mathbb N$ and $x\in X$ we write:
$$
\lim_{n,{\mathcal F}}x_n=x
$$
if for every $W\in \mathcal N(X)$ one has $\{n\in \N: x_n-x\in W\}\in \mathcal F$. Since $X$ is compact Hausdorff, it follows that for every $\mathcal F\in \mathfrak F$ and $(x_n)_{n\in \mathbb N}\in {X}^\mathbb N$ there exists a unique $x\in X$ such that $\lim_{n,{\mathcal F}}x_n=x$.

 For a filter  $\mathcal F\in \mathfrak F$ define the mapping $\chi_{{\mathcal F}}:{ X}^\mathbb N\to X$ by the equality:
$$
\chi_{\mathcal F}({\bf x})=\lim_{n,{\mathcal F}}x_n,\quad \forall {\bf x}=(x_n)_{n\in \mathbb N}\in {X}^\mathbb N\,.
$$
 Then
 \begin{equation}\label{2sep}
 \chi_{\mathcal F}\in CHom (X^{\N},X)\quad \forall \mathcal F\in \mathfrak F\,.
 \end{equation}
 To verify (\ref{2sep}), fix $\mathcal F\in \mathfrak F\,.$ As
$$
\chi_{\mathcal F}({\bf x}+{\bf y})=\lim_{n,{\mathcal F}}(x_n+y_n)= \lim_{n,{\mathcal F}}x_n+\lim_{n,{\mathcal
F}}y_n=\chi_{\mathcal F}({\bf x})+\chi_{\mathcal F}({\bf y}),\quad \forall {\bf x},{\bf y}\in {X}^\mathbb N\,,
$$
we conclude that $\chi_{\mathcal F}\in Hom (X^{\N},X)$. To see that $\chi$ is continuous on  $(X^{\mathbb N},\mathfrak{u})$, fix a
closed $W\in \mathcal N(X)$. Since $W$ is closed, for ${\bf x}=(x_n)_{n\in \mathbb N}\in {W}^\mathbb N$ we shall have
$$
\chi_{\mathcal F}({\bf x})=\lim_{n,{\mathcal F}}x_n\in W\,.
$$
Consequently, $\chi_{\mathcal F}(W^{\mathbb N})\subset W$. From this relation, as $W^{\mathbb N}\in \mathcal N(X^{\mathbb N},\mathfrak{u})$,
we get that $\chi_{\mathcal F}$ is continuous on  $(X^{\mathbb N},\mathfrak{u})$ and (\ref{2sep}) is proved.
\par
We have also:
\begin{equation}\label{2sep2}
 \mathcal F_1\in \mathfrak F,\,\mathcal F_2\in
\mathfrak F,\,\,\mathcal F_1\ne \mathcal F_2\,
\Longrightarrow\,\chi_{\mathcal F_1} \ne \chi_{\mathcal F_2}
\end{equation}
 In fact, as $\mathcal F_1$ and $\mathcal F_2$ are {\it distinct} ultrafilters, there is $F\in  \mathcal F_1$ such that $F\not\in  \mathcal F_2$. Let  ${\bf x}=(x_n)_{n\in \mathbb N}\in
{X}^\mathbb N\,$ be defined by conditions: $x_n=0$ if $n\in F$ and $x_n=a\ne 0$ if $n\in \mathbb N\setminus F$. Then $\chi_{\mathcal
F_1}({\bf x})=0$ and $\chi_{\mathcal F_2}({\bf x})=a$. Therefore, $\chi_{\mathcal F_1} \ne \chi_{\mathcal F_2}$ and (\ref{2sep2}) is
proved.
 \par
Clearly (\ref{2sep0}),(\ref{2sep})  and (\ref{2sep2}) imply that ${\rm Card}\ CHom (X^{\N},X) \ge 2^{\rm{c}}$.
\end{proof}

\section{Open questions}

It follows from \cite[Proposition 5.4]{CMPT} that every {\bf non-meager}  $(G,\mu)\in \rm{LCS}$ is a Mackey group in $\rm{LQC}$.

\begin{conj}\label{upki}{\em Every metrizable  $(G,\mu)\in \rm{LCS}$ is a Mackey group in $\rm{LQC}$. }
\end{conj}
We do not even  know if $\R^{(\N)}$ with the topology induced from the product space $\R^{\N}$ is a Mackey group in $\rm{LQC}$.

\begin{remark}\label{oques}{\em
We conjecture that Proposition \ref{elenaN1}$(c)$ remains true for a (not necessarily complete metrizable) topological abelian group.  }
  \end{remark}

 It is clear that if $G$ is a discrete group, then $G$  is a Mackey group in $\rm{MAP}$.

\begin{conj}\label{conj1} If $G\in \rm{MAP}$ is a Mackey group in $\rm{MAP}$, then $G$ is a discrete group.
\end{conj}

 The conjecture \ref{conj1} in terms of the Mackey topology can be reformulated as follows.

 \begin{conj}\label{conj2}  If for a precompact topological group $(G,\nu)$ there exists  the $\rm{MAP}$-Mackey topology in $G$ associated with $\nu$, then $\nu$
 is the finest precompact group topology in $G$.
\end{conj}

\begin{remark}{\em
In \cite{BTM} it is shown that a non-complete precompact metrizable $(G,\mu)$  can be a  Mackey group in $\rm{LQC}$. It is also known that every locally pseudocompact
$(G,\mu)$ is a Mackey group in $\rm{LQC}$ (cf. \cite{DLMT}). An internal description of groups $(G,\tau)\in$$\rm{LQC}$ which are Mackey in $\rm{LQC}$ is unknown.
}
\end{remark}


\end{document}